\documentclass[a4paper,11pt]{amsart}

\usepackage[margin=1in]{geometry}
\usepackage{graphicx}
\usepackage{amsmath}
\usepackage[utf8]{inputenc}
\usepackage{amsfonts}
\usepackage{amsthm}
\usepackage{xcolor}
\usepackage[pagebackref=true,pdftex,bookmarksopen=true,colorlinks=true, linkcolor=red, citecolor=green]{hyperref}

\usepackage{marginnote}

\numberwithin{equation}{section}

\newtheorem{thm}{Theorem}[section]

\newtheorem{lemma}[thm]{Lemma}
\newtheorem{prop}[thm]{Proposition}

\newtheorem{remark}[thm]{Remark}
\newtheorem{definition}{Definition}[section]

\newcommand{\R}{\mathbb{R}}

\newcommand{\C}{\mathbb{C}}
\newcommand{\N}{\mathbb{N}}
\newcommand{\Z}{\mathbb{Z}}

\newcommand{\X}{\mathcal{X}}
\newcommand{\HH}{\mathcal{H}}
\newcommand{\ZZ}{\mathcal{Z}}

\newcommand{\FF}{\mathcal{F}}
\newcommand{\NN}{\mathcal{N}}

\begin{document}

\title[Controllability Gear--Grimshaw with Critical Size Restrictions]{Neumann Boundary Controllability of the Gear--Grimshaw System With Critical Size Restrictions on the Spacial Domain}

\author[Capistrano--Filho]{Roberto A. Capistrano--Filho}

\address{%
Department of Mathematics\\
Federal University of Pernambuco, UFPE, \\
CEP 50740-545, Recife, PE\\
Brazil.}

\email{capistranofilho@dmat.ufpe.br}

\author[Gallego]{Fernando A. Gallego}

\address{%
Institute of Mathematics \\
Federal University of Rio de Janeiro, UFRJ\\
P.O. Box 68530, CEP 21945-970, Rio de Janeiro, RJ\\
Brazil.}

\email{fgallego@ufrj.br, ferangares@gmail.com}

\author[Pazoto]{Ademir F. Pazoto}

\address{%
Institute of Mathematics \\
Federal University of Rio de Janeiro, UFRJ\\
P.O. Box 68530, CEP 21945-970, Rio de Janeiro, RJ\\
Brazil.}

\email{ademir@im.ufrj.br}



\date{}

\begin{abstract}
In this paper we study the boundary controllability of the Gear-Grimshaw system posed on a finite domain $(0,L)$, with Neumann boundary conditions:
\begin{equation}
\label{abs}
\begin{cases}
u_t + uu_x+u_{xxx} + a v_{xxx} + a_1vv_x+a_2 (uv)_x =0, & \text{in} \,\, (0,L)\times (0,T),\\
c v_t +rv_x +vv_x+abu_{xxx} +v_{xxx}+a_2buu_x+a_1b(uv)_x  =0,  & \text{in} \,\, (0,L)\times (0,T), \\
u_{xx}(0,t)=h_0(t),\,\,u_x(L,t)=h_1(t),\,\,u_{xx}(L,t)=h_2(t), & \text{in} \,\, (0,T),\\
v_{xx}(0,t)=g_0(t),\,\,v_x(L,t)=g_1(t),\,\,v_{xx}(L,t)=g_2(t), & \text{in} \,\, (0,T),\\
u(x,0)= u^0(x), \quad v(x,0)=  v^0(x), & \text{in} \,\, (0,L).\nonumber
\end{cases}
\end{equation}
We first prove that the corresponding linearized system around the origin is exactly controllable in $(L^2(0,L))^2$ when $h_2(t)=g_2(t)=0$. In this case, the exact controllability property is derived for any $L>0$ with control functions $h_0, g_0\in H^{-\frac{1}{3}}(0,T)$ and $h_1, g_1\in L^2(0,T)$. If we change the position of the controls and consider $h_0(t)=h_2(t)=0$ (resp. $g_0(t)=g_2(t)=0)$ we obtain the result with control functions $g_0, g_2\in H^{-\frac{1}{3}}(0,T)$ and $h_1, g_1\in L^2(0,T)$ if and only if the length $L$  of the spatial domain $(0,L)$ belongs to a countable set. In all cases the regularity of the controls are sharp in time.
If only one control act in the boundary condition, $h_0(t)=g_0(t)=h_2(t)=g_2(t)=0$ and $g_1(t)=0$ (resp. $h_1(t)=0$), the linearized system is proved to be exactly controllable for small values of the length $L$ and large time of control $T$. Finally, the nonlinear system is shown to be locally exactly controllable \textit{via} the contraction mapping principle, if the associated linearized systems are exactly controllable.
\end{abstract}

\maketitle

\section{Introduction}

\subsection{Setting of the Problem}
The goal of this paper is to investigate the boundary controllability properties of the nonlinear dispersive system
\begin{equation}
\label{gg1}
\begin{cases}
u_t + uu_x+u_{xxx} + a v_{xxx} + a_1vv_x+a_2 (uv)_x =0, & \text{in} \,\, (0,L)\times (0,T),\\
c v_t +rv_x +vv_x+abu_{xxx} +v_{xxx}+a_2buu_x+a_1b(uv)_x  =0,  & \text{in} \,\, (0,L)\times (0,T), \\
u(x,0)= u^0(x), \quad v(x,0)=  v^0(x), & \text{in} \,\, (0,L),
\end{cases}
\end{equation}
with the following boundary conditions
\begin{equation}\label{gg2}\begin{cases}
u_{xx}(0,t)=h_0(t),\,\,u_x(L,t)=h_1(t),\,\,u_{xx}(L,t)=h_2(t),\\
v_{xx}(0,t)=g_0(t),\,\,v_x(L,t)=g_1(t),\,\,v_{xx}(L,t)=g_2(t).
\end{cases}
\end{equation}
 In \eqref{gg1}, $a_1, a_2, a, b, c$ and $r$ are real constants, $u=u(x,t)$ and $v=v(x,t)$ are real-valued functions of the two variables $x$ and $t$ and subscripts indicate partial differentiation. The boundary functions $h_i$ and $g_i$, for $i=0,1,2$, are considered as control inputs acting on the boundary conditions. Our purpose is to see weather we can force the solutions of the system to have certain properties by choosing appropriate control inputs. More precisely, we are mainly concerned with the following exact control problem:
\vglue 0.2 cm
 {\em  Given $T>0$ and $u^0,v^0,u^1,v^1\in L^2(0,L)$, can one find appropriate control inputs $h_i$, $g_i$, for $i=0,1,2$,  such that the corresponding solution $(u,v)$ of \eqref{gg1}-\eqref{gg2} satisfies
\begin{equation}\label{exactcontrol'a}
(u(x,T),v(x,T))=(u^1(x),v^1(x))?
\end{equation}}

In order to provide the tools to handle with this problem, we assume that the coefficients $a, b, c$ and $r$ satisfy
\begin{equation}\label{coef}
b, c \mbox{ and } r \mbox{ are positive and } 1-a^2b >0.
\end{equation}

System \eqref{gg1} was derived by Gear and Grimshaw in \cite{geargrimshaw1984} as a model to describe strong interactions of two long internal gravity waves in a stratified fluid, where the two waves are assumed to correspond to different modes of the linearized equations of motion (we also refer to \cite{BoPoSaTo,SaTz} for an extensive discussion on the physical relevance of the system). This somewhat complicated model has the structure of a pair of Korteweg-de Vries (KdV) equations coupled through both dispersive and nonlinear effects and has been object of intensive research in recent year. It is a special case of a broad class of nonlinear evolution equations for which the well-posedness theory associated to the pure initial-value problem posed on the whole real line $\mathbb{R}$, or on a finite interval with periodic boundary conditions, has been intensively investigated. By contrast, the mathematical theory pertaining to the study of the boundary value problem is considerably less advanced, specially in what concerns the study of the controllability properties. As far as we know, the controllability results for system \eqref{gg1} was first obtained in \cite{MiOr}, when the model is posed on a periodic domain and $r=0$. In this case, a diagonalization of the main terms allows to the decouple the corresponding linear system into two scalar KdV equations and use the previous results available in the literature. Later on,  assuming that \eqref{coef} holds,  Micu \textit{et al.} \cite{micuortegapazoto2009} proved the local exact boundary controllability property for the nonlinear system, posed on a bounded interval, considering the following boundary conditions:
\begin{equation}\label{gg3}\begin{cases}
u(0,t)=0,\,\,u(L,t)=f_1(t),\,\,u_x(L,t)=f_2(t),\\
v(0,t)=0,\,\,v(L,t)=k_1(t),\,\,v_x(L,t)=k_2(t).
\end{cases}
\end{equation}
The analysis developed in \cite{micuortegapazoto2009} was inspired by the results obtained by Rosier in \cite{rosier} for the scalar KdV equation. It combines the analysis of the linearized system and the Banach's fixed point theorem. Following the classical duality approach \cite{dolecki,lions}, the exact controllability of system linearized system is equivalent to an observability for the adjoint system. Then, the problem is reduced to prove a nonstandard unique continuation property of the eigenfunctions of the corresponding differential operator. Their main result reads as follows:
\vglue 0.2 cm
\noindent
{\bf Theorem A }(Micu \textit{et al.} \cite{micuortegapazoto2009}) {\em Let $L>0$ and $T >0$. Then there exists a constant $\delta>0$, such that, for any initial and final data $u^0, v^0, u^1, v^1\in L^2(0, L)$ verifying $$||(u^0, v^0)||_{(L^2(0,L))^2}\leq\delta \quad\text{and}\quad ||(u^1, v^1)||_{(L^2(0,L))^2}\leq\delta,$$ there exist four control functions $f_1, k_1\in H^1_0(0, T)$ and $f_2, k_2 \in L^2(0,T)$, such that the solution $$(u, v)\in C([0,T];(L^2(0,L))^2)\cap L^2(0,T;(H^1(0, L))^2)\cap H^1(0,T;(H^{-2}(0, L))^2)$$ of \eqref{gg1}-\eqref{gg3} verifies \eqref{exactcontrol'a}.}

Later on, the same problem was addressed by Cerpa and Pazoto \cite{cerpapazoto2011} when only two controls act on the Neumann boundary conditions, i.e., assuming that $f_1=k_1=0$. In this case, the analysis of the linearized system is much more complicated, therefore the authors used a direct approach based on the multiplier technique that gives the observability inequality for small values of the length $L$ and large time of control $T$. The fixed point argument, as well as, the existence and regularity results needed in order to consider the nonlinear system run exactly in the same way as in \cite{micuortegapazoto2009}.

The program of this work was carried out for a particular choice of boundary conditions
and aims to establish as a fact that such a model predicts the interesting controllability properties initially
observed for the KdV equation. Therefore, to introduce the reader to the theory developed for KdV with the boundary conditions of types \eqref{gg3} and \eqref{gg2}, we present below a summary of the results achieved in \cite{rosier} and \cite{caicedo_caspistrano_zhang_2015}, respectively.

Rosier, in \cite{rosier}, studied the following boundary control problem for the KdV equation posed on the finite domain $(0,L)$
\begin{equation}
\left\{
\begin{array}
[c]{lll}%
u_t+u_x+uu_x+u_{xxx}=0 &  & \text{ in } (0,L)\times(0,T),\\
u(0,t)=0,\text{ }u(L,t)=0,\text{ }u_x(L,t)=g(t) &  & \text{ in }(0,T),\\
u(x,0)=u^0(x) & & \text{ in }(0,L),
\end{array}
\right.  \label{2}
\end{equation}
where the boundary value function $g(t)$ is considered as a control input. First, the author studies the associated linear system
\begin{equation}
\left\{
\begin{array}
[c]{lll}%
u_t+u_x+u_{xxx}=0 &  & \text{ in } (0,L)\times(0,T),\\
u(0,t)=0,\text{ }u(L,t)=0,\text{ }u_x(L,t)=g(t) &  & \text{ in }(0,T),\\
u(x,0)=u^0(x) & & \text{ in }(0,L)
\end{array}
\right.  \label{2a}
\end{equation}
and discovered the so-called {\em critical length} phenomena, i.e., whether the system (\ref{2a}) is exactly controllable depends on the length $L$ of the spatial domain $(0,L)$. More precise, the following result was proved:

\smallskip
 \noindent
 {\bf Theorem B} (Rosier \cite{rosier})  {\em The linear system  \eqref{2a}  is exactly controllable in the space  $L^2(0,L)$ if and only if the length $L$  of the spatial domain $(0,L)$ does not belong to the set}
\begin{equation}
\mathcal{N}:=\left\{  \frac{2\pi}{\sqrt{3}}\sqrt{k^{2}+kl+l^{2}}
\,:k,\,l\,\in\mathbb{N}^{\ast}\right\}  . \label{critical}
\end{equation}

\smallskip
  Then, by using a fixed point argument, the controllability result was extended to the nonlinear system when $L\notin\mathcal{N}$.
\vglue 0.2 cm
\noindent
{\bf Theorem C} (Rosier \cite{rosier})\textit{Let $T>0$ be given. If $L\notin\mathcal{N}$, there exists $\delta>0$, such that, for any $u^0,u^T\in L^2(0,L)$ with $$||u^0||_{L^2(0,L)}+||u^T||_{L^2(0,L)}\leq\delta,$$ one can  find a  control input $g\in L^2(0,T)$, such that the  nonlinear system \eqref{2} admits a unique solution $$u\in C([0,T];L^2(0,L))\cap L^2(0,T;H^1(0,L))$$
satisfying $$u(x,T)=u^T(x).$$}

More recently, in \cite{caicedo_caspistrano_zhang_2015}, Caicedo \textit{et al.} investigated the boundary control problem of the KdV equation with new boundary conditions, namely, the Neumann boundary conditions:
\begin{equation}
\left\{
\begin{array}
[c]{lll}%
u_t+ (1+\beta )u_x+u_{xxx}=0 &  & \text{ in } (0,L)\times(0,T),\\
u_{xx}(0,t)=0,\text{ }u_x(L,t)=h(t),\text{ }u_{xx}(L,t)=0 &  & \text{ in }(0,T),\\
u(x,0)=u^0(x) & & \text{ in }(0,L).
\end{array}
\right.  \label{gg5}
\end{equation}
In \eqref{gg5}, $\beta $ is a given real constant and $g$ a control input. For any $\beta \ne -1$, the authors obtained the following set of
{\em critical lengths}
\begin{equation}
\mathcal{R}_{\beta} :=\left\{  \frac{2\pi}{\sqrt{3(1+\beta)}}\sqrt{k^{2}+kl+l^{2}}
\,:k,\,l\,\in\mathbb{N}^{\ast}\right\}\cup\left\{\frac{k\pi}{\sqrt{\beta +1}}:k\in\mathbb{N}^{\ast}\right\},
\label{critical_new}
\end{equation}
and proved that the following result holds:
\vglue 0.2 cm
\noindent
{\bf Theorem D} (Caicedo \textit{et al.} \cite{caicedo_caspistrano_zhang_2015}){\em
\begin{itemize}
\item[(i)] If $\beta \ne -1$, the linear system (\ref{gg5}) is exactly controllable in the space $L^2 (0,L)$  if and only if the length L of the spatial domain $(0, L)$ does not belong to the set $\mathcal{R}_{\beta}$.
\item[(ii)] If $\beta =-1$, then the system (\ref{gg5}) is not  exact controllable in the space $L^2 (0,L)$ for any $L>0$.
\end{itemize}}
In addition, for the nonlinear system
\begin{equation}
\left\{
\begin{array}
[c]{lll}%
u_t+u_x+uu_x+u_{xxx}=0 &  & \text{ in } (0,L)\times(0,T),\\
u_{xx}(0,t)=0,\text{ }u_x(L,t)=h(t),\text{ }u_{xx}(L,t)=0 &  & \text{ in }(0,T),\\
u(x,0)=u_0(x) & & \text{ in }(0,L),
\end{array}
\right.  \label{gg6}
\end{equation}
the result below was proved by using a fixed point argument:
\vglue 0.2 cm
\noindent
{\bf Theorem E} (Caicedo \textit{et al.} \cite{caicedo_caspistrano_zhang_2015}) {\em
Let $T>0$,  $\beta \ne -1$ and $L\notin\mathcal{R}_{\beta} $ be given.  There exists $\delta>0$, such that, for any $u^0,u^T\in L^2(0,L)$ with $$||u^0-\beta ||_{L^2(0,L)}+||u^T-\beta ||_{L^2(0,L)}\leq\delta,$$ one can find  a control input $h\in L^2(0,T)$, such that the system \eqref{gg6} admits unique solution
$$u\in C([0,T];L^2(0,L))\cap L^2(0,T;H^1(0,L))$$
satisfying
$$u(x,T)=u^T(x).$$}

Both theorems, Theorems B and D, were proved following the classical duality approach \cite{dolecki,lions} which reduces the problem to prove an observability inequality for the solutions of the corresponding adjoint system. Then, the controllability is obtained with the aid of a compactness argument that leads the issue to a nonstandard unique continuation principle for the eigenfunctions of the differential operator associated to the model. The critical lengths in \eqref{critical} and \eqref{critical_new} are such that there are eigenfunctions of the linear scalar problem for which the observability inequality associated to the adjoint system fails\footnote{In the case of $L\in\mathcal{N}$ (resp. $L\in\mathcal{R}_{\beta}$), Rosier (resp. Caicedo \textit{et al.} in \cite{caicedo_caspistrano_zhang_2015}) proved in \cite{rosier} that the associated linear system \eqref{2a} is not controllable; there exists a finite-dimensional subspace of $L^2(0,L)$, denoted by $\mathcal{M}=\mathcal{M}(L)$, which is unreachable from $0$ for the linear system.  More precisely, for every nonzero state $\psi\in\mathcal{M}$, $g\in L^2(0,T)$  and $u\in C([0,T];L^2(0,L))\cap L^2(0,T;H^1(0,L))$ satisfying \eqref{2a} and $u(\cdot,0)=0$, one has $u(\cdot,T)\neq\psi$.
A spatial domain $(0,L)$ is called \textit{critical}  for the system (\ref{2a}) (resp. \eqref{gg5}) if its domain length $L\in\mathcal{N}$ (resp. $L\in\mathcal{R}_{\beta}$).}.  However, in \cite{caicedo_caspistrano_zhang_2015}, the authors encountered some difficulties that require special attention. For instance, the adjoint system of the linear system \eqref{gg5} is given by
\begin{equation}
\left\{
\begin{array}
[c]{lll}%
\psi_t+(1+\beta )\psi_x+\psi_{xxx}=0 &  &\text{ in } (0,L)\times(0,T),\\
(1+\beta )\psi(0,t)+\psi_{xx}(0,t)=0&  &\text{ in } (0,T),\\
(1+\beta )\psi(L,t)+\psi_{xx}(L,t)=0&  &\text{ in } (0,T),\\
\psi_x(0,t)=0&  &\text{ in } (0,T),\\
\psi(x,T)=\psi^T(x) & &\text{ in }(0,L).
\end{array}
\right.  \label{gg6_a}
\end{equation}
The exact controllability of system \eqref{gg5} is equivalent to the following observability inequality for the adjoint system \eqref{gg6_a}:
\begin{equation*}
||\psi^T||_{L^2(0,L)}\leq C||\psi_x(L,\cdot)||_{L^2(0,T)},
\end{equation*}
for some $C>0$. Nonetheless, the usual multiplier method and compactness arguments used to deal with the system \eqref{gg6_a} only lead to
\begin{equation}
||\psi^T||^2_{L^2(0,L)}\leq C_1||\psi_x(L,\cdot)||^2_{L^2(0,T)}+C_2||\psi(L,\cdot)||^2_{L^2(0,T)},\label{gg7_a}
\end{equation}
where $C_1$ and $C_2$ are positive constants. In order to absorb the extra term present in \eqref{gg7_a}, Caicedo \textit{et al.}  derived a technical result, which reveals some hidden regularity (sharp trace regularities) for solutions of the adjoint system \eqref{gg6_a}:
\vglue 0.2 cm
\noindent
{\bf Theorem F} (Caicedo \textit{et al.} \cite{caicedo_caspistrano_zhang_2015}) {\em For any $\psi^T\in L^2(0,L)$, the solution $$\psi\in C([0,T];L^2(0,L))\cap L^2(0,T;H^1(0,L))$$ of the problem \eqref{gg6_a} possesses the following sharp trace properties
\begin{equation}
\sup_{x\in(0,L)}||\partial^r_x\psi(x,\cdot)||_{H^{\frac{1-r}{3}}(0,T)}\leq C_r||\psi^T||_{L^2(0,L)},
\label{9}
\end{equation}
for $r=0,1,2$, where $C_r$ are positive constants.}

\vglue 0.2 cm

\noindent Estimate \eqref{9} is then combined with compactness argument to remove the extra term in \eqref{gg7_a}. We remark that the sharp Kato smoothing properties obtained by Kenig, Ponce and Vega \cite{KePoVe} for the solutions of the KdV equation posed on the line, played an important role in the proof of the previous result. The same strategy has been successfully applied by Cerpa \textit{et al.} \cite{cerizh} for the study of a similar boundary controllability problem.

\subsection{Main Result}

We are now in position to return considerations to the control properties of the system \eqref{gg1}. First, we prove that the corresponding linear system
with the following boundary conditions
\begin{equation*}
\begin{cases}
u_{xx}(0,t)=h_0(t),\,\,u_x(L,t)=h_1(t),\,\,u_{xx}(L,t)=0,\\
v_{xx}(0,t)=g_0(t),\,\,v_x(L,t)=g_1(t),\,\,v_{xx}(L,t)=0,
\end{cases}
\end{equation*}
is exactly controllable in $(L^2(0,L))^2$ with controls $h_0$, $g_0\in H^{-\frac{1}{3}}(0,T)$ and $h_1$, $g_1\in L^2(0,T)$. In this case, any restriction on the length $L$ of the spatial domain is required. However, if we change the position of the controls a critical size restriction can appear. This is the case when we consider the following boundary conditions
\begin{equation*}
\begin{cases}
u_{xx}(0,t)=0,\,\,u_x(L,t)=h_1(t),\,\,u_{xx}(L,t)=0,\\
v_{xx}(0,t)=g_0(t),\,\,v_x(L,t)=g_1(t),\,\,v_{xx}(L,t)=g_2(t).
\end{cases}
\end{equation*}
In this case, the exact controllability result in $(L^2(0,L))^2$ is derived with controls $g_0$, $g_2\in H^{-\frac{1}{3}}(0,T)$ and $h_1$, $g_1\in L^2(0,T)$ if and only if the length $L$ does not belong of the following set

\begin{equation}
\FF_r:= \left\lbrace 2\pi k \sqrt{\frac{1-a^2b}{r}}: k \in \N^{*}\right\rbrace\cup \left\lbrace \pi \sqrt{\frac{(1-a^2b)\alpha(k,l,m,n,s)}{3r}}:k, l, m, n, s \in \N^{*} \right\rbrace,
\label{critical_f}
\end{equation}
where
\begin{align*}
\alpha:=\alpha(k,l,m,n,s)=&5k^2+8l^2+9m^2+8n^2+5s^2+8kl+6km\\
&+4kn+2ks+12ml+8ln+3ls+12mn+6ms+8ns.
\end{align*}
As in \cite{caicedo_caspistrano_zhang_2015}, the hidden regularity for the corresponding adjoint system \eqref{gg1} was required. Here, the result is given in Proposition \ref{hiddenregularities}, which is the key point to prove the controllability result.

Finally, for small values of the length $L$ and large time of control $T$ we derive a exact controllability result in $(L^2(0,L))^2$ by assuming that the controls $g_1(t)=0$ (resp. $h_1(t)=0$) and $g_0(t)=g_2(t)=0$. In this case, the analysis of the linearized system is much more complicated, therefore we use a direct approach based on the multipliers technique, as in \cite{cerpapazoto2011}. In all cases, the result obtained for the linear system allows to prove the local controllability property of the nonlinear system \eqref{gg1} by means of a fixed point argument.

The analysis describe above are summarized in the main result of the paper, Theorem \ref{main_theo}. However, in order to make the reading easier, throughout the paper we use the following notation for the boundary functions:
\begin{itemize}
\item[] $\vec{h}_1=(0,h_1,0),\,\, \vec{g}_1=(g_0,g_1,g_2)$\,\,  and \,\, $\vec{h}_2=(h_0,h_1,h_2),\,\, \vec{g}_2=(0,g_1,0)$,
\item[] $\vec{h}_3=(h_0,h_1,0),\,\,\vec{g}_3=(g_0,g_1,0)$\,\, and \,\, $\vec{h}_4=(0,h_1,h_2),\,\,\vec{g}_4=(0,g_1,g_2)$,
\item[] $\vec{h}_5=(0,h_1,0),\,\,\vec{g}_5=(0,0,0)$\,\,\,\,\,\,\,\,\,\, and \,\, $\vec{h}_6=(0,0,0),\,\,\vec{g}_6=(0,g_1,0)$.
\end{itemize}
We also introduce the space $\mathcal{X}:=(L^2(0,L))^2$ endowed with the inner product
\begin{equation*}
\left\langle (u,v) , (\varphi
,\psi)\right\rangle := \int_0^L u(x)\varphi(x) dx + \frac{b}{c}\int_0^L v(x)\psi(x) dx,\qquad \forall (u,v), (\varphi,\psi) \in \X,
\end{equation*}
and the spaces
$$\HH_T:=H^{-\frac{1}{3}}(0,T)\times L^2(0,T)\times H^{-\frac{1}{3}}(0,T) \mbox{ and } \ZZ_T:= C([0,T];(L^2(0,L))^2)\cap L^2(0,T,(H^1(0,L))^2)$$
endowed with their natural inner products.

Thus, our main result reads as follows:
\begin{thm}\label{main_theo}
Let $T>0$. Then, there exists $\delta>0$, such that,  for  any $(u^0,v^0), (u^1,v^1) \in \X:=(L^2(0,L))^2$ verifying
$$\|(u^0,v^0)\|_{\X} + \|(u^1,v^1)\|_{\X} \leq \delta,$$
the following holds:
\begin{enumerate}
\item[(i)] If $L \in (0,\infty) \setminus \FF_r$, one can find  $\vec{h}_i, \vec{g}_i \in \HH_T$, for $i=1,2$, such that the system \eqref{gg1}-\eqref{gg2} admits a unique solution $(u,v) \in \ZZ_T$ satisfying \eqref{exactcontrol'a}.
\item[(ii)] For any $L>0$, one can find  $\vec{h}_i, \vec{g}_j \in \HH_T$, for $j=3,4$, such that the system \eqref{gg1}-\eqref{gg2} admits a unique solution $(u,v) \in \ZZ_T$, satisfying \eqref{exactcontrol'a}.
\item[(iii)]  Let $T>0$ and $L>0$ satisfying
\begin{align*}
1>\frac{\beta C_T}{T}\left[L +\frac{r}{c} \right],
\end{align*}
where $C_T$ is the constant in \eqref{hr4} and $\beta$ is the constant given by the embedding $H^{\frac{1}{3}}(0,T) \subset L ^2(0,T)$. Then, one can find  $\vec{h}_k, \vec{g}_k \in \HH_T$, for $k=5,6$, such that the system \eqref{gg1}-\eqref{gg2} admits a unique solution $(u,v) \in \ZZ_T$, satisfying \eqref{exactcontrol'a}.
\end{enumerate}
 \end{thm}

Before close this section, we observe that the exact controllability result given in Theorem A holds without any restriction of the Length $L$. However, we believe that, with another configuration of the controls, it is possible to prove the existence of a critical set for the system \eqref{gg1}.

\medskip

The article is organized as follows:

\medskip

----  In Section \ref{Sec2},  we show that  the system \eqref{gg1}-\eqref{gg2} is locally well-posed in $\ZZ_T$,  whenever  $(u^0, v^0) \in(L^2 (0,L))^2$, $h_0, \ g_0\in H^{-\frac13} (\mathbb{R}^+), \ h_1, \ g_1 \in L^2 (\mathbb{R}^+)$ and $h_2, \ g_2\in H^{-\frac13} (\mathbb{R}^+)$.  Various linear estimates, including hidden regularities,  are  presented for solutions of the corresponding linear system. As we pointed before, such estimates will play important roles in studying the controllability properties.

\medskip

---- In Section \ref{Sec3},  the boundary control system \eqref{gg1} is investigated for its controllability.  We investigate  first the linearized system
and its corresponding adjoint system for their controllability and observability. In particular,  the hidden regularities for the solutions of the adjoint system  presented in the Section \ref{Sec2} are used to prove observability inequalities associated to the control problem.

\medskip

---- Finally, the proof of our main result, Theorem \ref{main_theo}, is presented in Section \ref{Sec4}.

\section{Well-posedness}\label{Sec2}
\subsection{Linear System}
In this section, we establish the well-posedness of the linear system associated to (\ref{gg1})-(\ref{gg2}):
\begin{equation}\label{gglin}
\left\lbrace \begin{tabular}{l l}
$u_t + u_{xxx} + av_{xxx}  =0$, & in $(0,L)\times (0,T)$,\\
$v_t +\frac{r}{c}v_x+\frac{ab}{c}u_{xxx} + \frac{1}{c}v_{xxx} =0$, & in $(0,L)\times (0,T)$,\\
$u_{xx}(0,t) = h_0(t),\,\,u_x(L,t) = h_1(t),\,\,u_{xx}(L,t) = h_2(t)$,& in $(0,T)$,\\
$v_{xx}(0,t) =g_0(t),\,\,v_x(L,t) = g_1(t),\,\, v_{xx}(L,t) = g_2(t)$,& in $(0,T)$,\\
$u(x,0)=u^0(x), \quad v(x,0) = v^0(x)$, & in $(0,L)$.
\end{tabular}\right.
\end{equation}
We begin by considering the following linear non-homogeneous boundary value problem
\begin{equation}\label{gglin1}
\left\lbrace \begin{tabular}{l l}
$u_t + u_{xxx} + av_{xxx}  =f$, & in $(0,L)\times (0,T)$,\\
$v_t +\frac{ab}{c}u_{xxx} + \frac{1}{c}v_{xxx} =s$, & in $(0,L)\times (0,T)$,\\
$u_{xx}(0,t) = h_0(t),\,\,u_x(L,t) = h_1(t),\,\,u_{xx}(L,t) = h_2(t)$,& in $(0,T)$,\\
$v_{xx}(0,t) =g_0(t),\,\,v_x(L,t) = g_1(t),\,\, v_{xx}(L,t) = g_2(t)$,& in $(0,T)$,\\
$u(x,0)=u^0(x), \quad v(x,0) = v^0(x)$, & in $(0,L)$,
\end{tabular}\right.
\end{equation}
 with the notation introduced in Section 1. Then next proposition shows that the, problem (\ref{gglin1}) is well-posed in the space $\X$.
\begin{prop}\label{prop1}
Let $T>0$ be given. Then, for any $(u^0,v^0)$ in $\X$, $f,s$ in $L^1(0,T;L^2(0,L))$ and $\overrightarrow{h},\overrightarrow{g}\in \HH_T$, problem (\ref{gglin1}) admits a unique solution $(u,v) \in \ZZ_T$, with
\begin{equation}\label{hr1}
\partial_x^k u,\partial_x^k v \in L^{\infty}_x(0,L;H^{\frac{1-k}{3}}(0,T)),  \quad k=0,1,2.
\end{equation}
Moreover, there exist $C>0$, such that
\begin{multline*}
\|(u,v)\|_{\ZZ_T}+\sum_{k=0}^{2}\|(u,v)\|_{L^{\infty}_x(0,L;(H^{\frac{1-k}{3}}(0,T))^2)}\leq C\left\lbrace \|(u^0,v^0)\|_{\X} \right. \\
\left.+\|(\overrightarrow{h},\overrightarrow{g})\|_{\HH_T}+\|(f,s)\|_{L^1(0,T;(L^2(0,L))^2)} \right\rbrace.
\end{multline*}
\end{prop}

\begin{proof}
We diagonalize the main term in \eqref{gglin} and consider the change of variable
\begin{equation*}
\left\lbrace\begin{tabular}{l}
$u = 2a \widetilde{u} +2a \widetilde{v}$, \\
$v = \left(\left(\frac{1}{c}-1\right)+\lambda\right) \widetilde{u}+ \left(\left(\frac{1}{c}-1\right)-\lambda\right) \widetilde{v}$,
\end{tabular}\right.
\end{equation*}
where $\lambda=\sqrt{\left(\frac{1}{c}-1\right)^2+\frac{4a^2b}{c}}$. Thus, we can transform the
linear system (\ref{gglin1}) into
\begin{equation}\label{kdvln1}
\left\{
\begin{array}{ll}\vspace{2mm}
\widetilde{u}_t + \alpha_{-}\widetilde{u}_{xxx} =\widetilde{f},\\
 \widetilde{v}_t +\alpha_{+}\widetilde{v}_{xxx}=\widetilde{s}, \vspace{2mm} \\
\widetilde{u}_{xx}(0,t) = \widetilde{h}_0(t), \  \widetilde{u}_x(L,t) = \widetilde{h}_1(t),  \  \widetilde{u}_{xx}(L,t)  = \bar{h}_2(t) ,\vspace{2mm} \\
\widetilde{v}_{xx}(0,t) = \widetilde{g}_0(t), \ \widetilde{v}_x(L,t)  =  \widetilde{g}_1(t), \ \widetilde{v}_{xx}(L,t)  = \widetilde{g}_2(t) ,\vspace{2mm} \\
 \widetilde{u}(x,0)= \widetilde{u}^0(x), \quad \widetilde{v}(x,0)  = \widetilde{v}^0(x),
\end{array}
\right.
\end{equation}
where $\alpha_{\pm} = -\frac{1}{2}\left(\left(\frac{1}{c}-1\right)\pm \lambda\right)$ and
\begin{equation*}
\left\lbrace\begin{tabular}{l l l l}
$\widetilde{f}=-\frac{1}{2}\left(\frac{\alpha_+}{a\lambda}f+\frac{1}{\lambda}s\right)$, & $\widetilde{u}_0=-\frac{1}{2}\left(\frac{\alpha_{-}}{a\lambda}u^0-\frac{1}{\lambda}v^0\right)$, & $\widetilde{h}_i=-\frac{1}{2}\left(\frac{\alpha_{-}}{a\lambda}h_i-\frac{1}{\lambda}g_i\right),$ & $i=0,1,2,$ \\
\\
$\widetilde{s}=-\frac{1}{2}\left(\frac{\alpha_{-}}{a\lambda}f-\frac{1}{\lambda}s\right),$ & $\widetilde{v}_0=\frac{1}{2}\left(\frac{\alpha_{+}}{a\lambda}u^0-\frac{1}{\lambda}v^0\right)$, & $\widetilde{g}_i=\frac{1}{2}\left(\frac{\alpha_{+}}{a\lambda}h_i-\frac{1}{\lambda}g_i\right),$  & $i=0,1,2.$
\end{tabular}\right.
\end{equation*}
Note that condition \eqref{coef} guarantees that $\alpha_{\pm}$ are nonzero. Therefore, system \eqref{kdvln1} can be decoupled into two single KdV equations as follows:
\begin{equation}
\left\{
\begin{array}{ll}
\widetilde{u}_t + \alpha_{-}\widetilde{u}_{xxx} =\widetilde{f},\vspace{2mm}\\
\widetilde{u}_{xx}(0,t) = \widetilde{h}_0(t), \  \widetilde{u}_x(L,t) = \widetilde{h}_1(t), \ \widetilde{u}_{xx}(L,t)  = \widetilde{h}_2(t),\vspace{2mm}\\
\widetilde{u}(0,x)= \widetilde{u}^0(x)
\end{array}
\right.
\label{kdv_ne_1}
\end{equation}
and
\begin{equation}
\left\{
\begin{array}{ll}
\widetilde{v}_t +\alpha_{+}\widetilde{v}_{xxx}=\widetilde{s},\vspace{2mm}  \\
\widetilde{v}_{xx}(0,t) = \widetilde{g}_0(t), \ \widetilde{v}_x(L,t)  =  \widetilde{g}_1(t), \ \widetilde{v}_{xx}(L,t)  = \widetilde{g}_2(t), \vspace{2mm}\\
\widetilde{v}(x,0)  = \widetilde{v}^0(x).
\end{array}
\right.
\label{kdv_ne_2}
\end{equation}
Here, we consider the solutions written on the form $\{W_{bdr}^{\pm}(t)\}_{t\geq 0}$ that will be called \textit{the boundary integral operator}. For this purpose we use a Lemma, which can be found in \cite[Lemma 2.4]{CaZhaSun}, for solutions of \eqref{kdv_ne_1} (or \eqref{kdv_ne_2}):
\begin{lemma}\label{l4-d}
The solution $u$ of the IBVP \eqref{kdv_ne_1} (or \eqref{kdv_ne_2}) can be written in the form
\[
u(x,t)=[{W}^{+}_{bdr}\vec{\widetilde{h}}](x,t):=[{W}^{+}_{bdr}\vec{h}](x,t):= \sum_{j,m=1}^3
[W^{+}_{j,m}h_{m}](x,t),
\]
where
\begin{equation}
 [W^{+}_{j,m}h](x,t) \equiv [U_{j,m} h](x,t) +\overline{[U_{j,m}h](x,t)}\label{2.2-4}
\end{equation}
with
\begin{equation} [U_{j,m} h](x,t)\equiv \frac{1}{2\pi } \int ^{+\infty }_{0} e^{i\rho ^3
t} e^{\lambda ^+_j (\rho ) x}   3\rho ^2
[Q_{j,m}^{+}h](\rho )d\rho \label{2.3-4}
\end{equation}
for $j=1,3, \ m=1,2,3 $ and
\begin{equation}
 [U_{2,m} h](x,t) \equiv  \frac{1}{2\pi } \int ^{+\infty }_{0} e^{i \rho ^3
 t} e^{-\lambda ^+_2 (\rho ) (1-x)}    3\rho ^2
[Q_{2,m}^{+}h](\rho )d\rho \label{2.4-4}
\end{equation}
for $m=1,2,3$.
Here
\begin{equation}\label{2.5-4} [Q_{j,m}
^{+}h] (\rho ):=\frac{\Delta ^{+}_{j,m} (\rho ) } {\Delta
^{+} (\rho )} \hat{h}   ^+ (\rho ), \qquad [Q_{2,m}^{+}h]
(\rho )=\frac{\Delta ^{+}_{2,m} (\rho ) } {\Delta ^{+} (\rho
)} e^{\lambda ^+_2 (\rho )} \hat{h} ^+ (\rho )\end{equation} for
$j=1,3 $ and $m=1,2,3$. Here $\hat{h}^+ (\rho) = \hat{h} (i\rho
^3)$, $\Delta
    ^{+}(\rho)$ and $\Delta ^{+}_{j,m} (\rho)$ are obtained from
    $\Delta 
    (s)$ and $\Delta _{j,m} (s)$ by replacing $s$ with $i \rho
    ^3  $ and  $\lambda _j ^+ (\rho)=\lambda _j (i \rho
    ^3 )$ where
    \[ \Delta = \lambda _1\lambda _2 \lambda _3\left (\lambda _1(\lambda _3-\lambda _2) e^{-\lambda _1}
    +\lambda _2(\lambda _1-\lambda _3) e^{-\lambda _2}+\lambda _3(\lambda _2-\lambda _1) e^{-\lambda
    _3}\right );\]
    \[ \Delta _{1,1} = e^{-\lambda _1}\lambda _2 \lambda _3(\lambda _3-\lambda
    _2), \  \Delta _{2,1} = e^{-\lambda _2}\lambda _1 \lambda _3(\lambda
    _1-\lambda _3), \  \Delta _{3,1} = e^{-\lambda _3}\lambda _1 \lambda _2(\lambda
    _2-\lambda _1);\]
    \[ \Delta _{1,2} = \lambda ^2_2 \lambda _3^2 (e^{\lambda _2} -
    e^{\lambda _3}), \ \Delta _{2,2} = \lambda _1^2\lambda ^2_3( e^{\lambda _3} -
    e^{\lambda _1}), \ \Delta _{3,2} = \lambda_1 ^2\lambda _2^2 ( e^{\lambda _1} -
    e^{\lambda _2});\]
    \[ \Delta _{1,3} = \lambda _2\lambda _3 (\lambda _2e^{\lambda _3}-\lambda
    _3
    e^{\lambda _2}), \ \Delta _{2,3} = \lambda _1 \lambda _3 (\lambda_3e^{\lambda _1}-\lambda
    _1
    e^{\lambda _3}), \ \Delta _{3,3} = \lambda _1 \lambda _2 (\lambda _1e^{\lambda _2}-\lambda
    _2
    e^{\lambda _1}).\]
\end{lemma}
\medskip

Since
\begin{equation*}
(\widetilde{u}^0, \widetilde{v}^0) \in \X,\quad (\widetilde{f},\widetilde{s}) \in L^1(0,T;(L^2(0,L))^2)\, \mbox{ and }\, \overrightarrow{\widetilde{h}}, \overrightarrow{\widetilde{g}} \in \HH_T,
\end{equation*}
by \cite[Proposition 2.5]{caicedo_caspistrano_zhang_2015}, we obtain the existence of $(\widetilde{u},\widetilde{v})\in \ZZ_T$, solution of the system \eqref{kdvln1}, such that
\begin{equation*}
\partial_x^k \widetilde{u},\partial_x^k \widetilde{v} \in L^{\infty}_x(0,L;H^{\frac{1-k}{3}}(0,T)),  \quad k=0,1,2,
\end{equation*}
and
\begin{multline*}
\|(\widetilde{u},\widetilde{v})\|_{\ZZ_T}+\sum_{k=0}^{2}\|(\widetilde{u},\widetilde{v})\|_{L^{\infty}_x(0,L;(H^{\frac{1-k}{3}}(0,T))^2)}\leq C\left\lbrace \|(\widetilde{u}^0,\widetilde{v}^0)\|_{\X}+\|(\overrightarrow{\widetilde{h}},\overrightarrow{\widetilde{g}})\|_{\HH_T}\right. \\
\left.+\|(\widetilde{f},\widetilde{s})\|_{L^1(0,T;(L^2(0,L))^2)} \right\rbrace,
\end{multline*}
for some constant $C>0$. Furthermore, we can write $\widetilde{u}$ and $\widetilde{v}$ in its integral form as follows
\begin{equation*}
\widetilde{u}(t)=W_0^{-}(t)\widetilde{u}^0+W_{bdr}^{-}(t)\overrightarrow{\widetilde{h}} + \int_0^tW_0^{-}(t-\tau)\widetilde{f}(\tau)d\tau,
\end{equation*}
\begin{equation*}
\widetilde{v}(t)=W_0^{+}(t)\widetilde{v}^0+W_{bdr}^{+}(t)\overrightarrow{\widetilde{g}} + \int_0^tW_0^{+}(t-\tau)\widetilde{s}(\tau)d\tau,
\end{equation*}
where $\{W_0^{\pm}(t)\}_{t\geq 0}$ are the $C_0$-semigroup in the space $L^2(0,L)$ generated by the linear operators $$A^{\pm}=-\alpha_{\pm}g''',$$ with domain $$D(A^{\pm})=\{g \in H^3(0,L): g''(0)=g'(L)=g''(L)=0\},$$
and $\{W_{bdr}^{\pm}(t)\}_{t\geq 0}$ are the operator given in Lemma \ref{l4-d} (see also \cite[Lemma 2.1]{caicedo_caspistrano_zhang_2015} for more details). Then, by change of variable we can easily verify that
\begin{equation*}
\begin{cases}
u(t)=W_0^{-}(t)u^0+W_{bdr}^{-}(t)\overrightarrow{h} + \displaystyle\int_0^tW_0^{-}(t-\tau)f(\tau)d\tau, \\
v(t)=W_0^{+}(t)v^0+W_{bdr}^{+}(t)\overrightarrow{g} + \displaystyle\int_0^tW_0^{+}(t-\tau)s(\tau)d\tau
\end{cases}
\end{equation*}
and the result follows.
\end{proof}

The global well-posedness of the  system (\ref{gglin}) is obtained using a fixed point argument.
\begin{prop}\label{prop2}
Let $T>0$ be given. Then, for any $(u^0,v^0) \in \X$ and $\overrightarrow{h}, \overrightarrow{g}\in \HH_T$, problem (\ref{gglin}) admits a unique solution
$(u,v) \in \ZZ_T$
with
\begin{equation*}
\partial_x^k u,\partial_x^k v \in L^{\infty}_x(0,L;H^{\frac{1-k}{3}}(0,T)),  \quad k=0,1,2.
\end{equation*}
Moreover, there exist $C>0$, such that
\begin{multline*}
\|(u,v)\|_{\ZZ_T}+\sum_{k=0}^{2}\|(u,v)\|_{L^{\infty}_x(0,L;(H^{\frac{1-k}{3}}(0,T))^2)}\leq C\left\lbrace \|(u^0,v^0)\|_{\X}+\|(\overrightarrow{h},\overrightarrow{g})\|_{\HH_T} \right.\\
\left.+\|(f,s)\|_{L^1(0,T;(L^2(0,L))^2)} \right\rbrace.
\end{multline*}
\end{prop}
\begin{proof}
Let $\FF_T:= \left\lbrace (u,v) \in \ZZ_T: (u,v) \in  L^{\infty}_x(0,L;(H^{\frac{1-k}{3}}(0,T))^2), k=0,1,2\right\rbrace$ equipped with the norm
\[
\|(u,v)\|_{\FF_T} = \|(u,v)\|_{\ZZ_T}+\sum_{k=0}^{2}\|(u,v)\|_{L^{\infty}_x(0,L;(H^{\frac{1-k}{3}}(0,T))^2)}.
\]
Let $0< \beta \leq T$ to be determined later. For each $u,v \in \FF_{\beta}$, consider the problem
\begin{equation}\label{e1}
\left\lbrace \begin{tabular}{l l}
$\omega_t + \omega_{xxx} + a\eta_{xxx}  =0$, & in $(0,L)\times (0,\beta)$,\\
$\eta_t +\frac{ab}{c}\omega_{xxx} + \frac{1}{c}\eta_{xxx} =-\frac{r}{c}v_x$, & in $(0,L)\times (0,\beta)$,\\
$\omega_{xx}(0,t) = h_0(t),\,\,\omega_x(L,t) = h_1(t),\,\,\omega_{xx}(L,t) = h_2(t)$,& in $(0,\beta)$,\\
$\eta_{xx}(0,t) =g_0(t),\,\,\eta_x(L,t) = g_1(t),\,\, \eta_{xx}(L,t) = g_2(t)$,& in $(0,\beta)$,\\
$\omega(x,0)=u^0(x), \quad v(x,0) = v^0(x)$, & in $(0,L)$.
\end{tabular}\right.
\end{equation}
According to Proposition \ref{prop1}, we can define the operator
\[
\Gamma: \FF_{\beta} \rightarrow \FF_{\beta}, \quad \text{given by} \quad \Gamma(u,v)=(\omega,\eta),
\]
where $(\omega,\eta)$ is the solution of (\ref{e1}). Moreover,
\begin{equation}\label{hr3}
\|\Gamma(u,v)\|_{\FF_\beta}\leq C\left\lbrace \|(u^0,v^0)\|_{\X}+\|(\overrightarrow{h},\overrightarrow{g})\|_{\HH_\beta}+\|(0,v_x)\|_{L^1(0,\beta;(L^2(0,L))^2)} \right\rbrace,
\end{equation}
where the positive  constant $C$ depends only on $T$. Since
$$ \|(0,v_x)\|_{L^1(0,\beta;L^2(0,L))}\leq \beta^{\frac{1}{2}}\|(u,v)\|_{\FF_\beta},$$
we obtain a positive constant $C>0$, such that
\begin{equation}\label{e2}
\|\Gamma(u,v)\|_{\FF_\beta}\leq C\left\lbrace \|(u^0,v^0)\|_{\X}+\|(\overrightarrow{h},\overrightarrow{g})\|_{\HH_\beta}\right\rbrace+C\beta^{\frac{1}{2}}\|(u,v)\|_{\FF_\beta}.
\end{equation}
Let $(u,v) \in B_{r}(0):=\left\lbrace (u,v) \in \FF_{\beta}: \|(u,v)\|_{\FF_{\beta}}\leq r\right\rbrace$, with $r=2C\left\lbrace \|(u^0,v^0)\|_{\X}+\|(\overrightarrow{h},\overrightarrow{g})\|_{\HH_\beta}\right\rbrace$.
Choosing $\beta>0$, satisfying
\begin{equation}\label{beta}
C\beta^{\frac{1}{2}}\leq \frac{1}{2},
\end{equation}
from (\ref{e2}) we obtain
$$\|\Gamma(u,v)\|_{\FF_\beta}\leq r.$$
The above estimate allows us to conclude that
\[
\Gamma: B_r(0)\subset\FF_{\beta} \rightarrow B_r(0).
\]
On the other hand, note that $\Gamma(u_1,v_1)-\Gamma(u_2,v_2)$ solves the following system
\begin{equation*}
\left\lbrace \begin{tabular}{l l}
$\omega_t + \omega_{xxx} + a\eta_{xxx}  =0$, & in $(0,L)\times (0,\beta)$,\\
$\eta_t +\frac{ab}{c}\omega_{xxx} + \frac{1}{c}\eta_{xxx} =-\frac{r}{c}(v_{1x}-v_{2x}) $, & in $(0,L)\times (0,\beta)$,\\
$\omega_{xx}(0,t) = \omega_x(L,t) = \omega_{xx}(L,t) = 0$,& in $(0,\beta)$,\\
$\eta_{xx}(0,t)=\eta_x(L,t) = \eta_{xx}(L,t) = 0$,& in $(0,\beta)$,\\
$\omega(x,0)=0, \quad v(x,0) = 0$, & in $(0,L)$.
\end{tabular}\right.
\end{equation*}
Again, from Proposition \ref{prop1} and \eqref{beta}, we have
\begin{equation*}
\begin{array}{l}
\|\Gamma(u_1,v_1)-\Gamma(u_2,v_2)\|_{\FF_\beta}\leq C\|(0,v_{1x}-v_{2x})\|_{L^1(0,\beta;(L^2(0,L))^2)}\leq C\beta^{\frac{1}{2}}\|(u_1,v_1)-(u_2,v_2)\|_{\FF_{\beta}}\\
\qquad\qquad\qquad\qquad\qquad\quad\leq \displaystyle\frac{1}{2}\|(u_1,v_1)-(u_2,v_2)\|_{\FF_{\beta}}.\nonumber
\end{array}
\end{equation*}
Hence, $\Gamma: B_r(0) \rightarrow B_r(0)$ is a contraction and, by Banach fixed point theorem, we obtain a unique $(u,v) \in B_r(0)$, such that $$\Gamma(u,v) = (u,v) \in \FF_{\beta},$$
 and \eqref{hr3} holds, for all $t \in (0,\beta)$. Since the choice of $\beta$ is independent of $(u^0,v^0)$, the standard continuation extension argument yields that the solution $(u,v)$ belongs to $\FF_T$. The proof is complete.\end{proof}
\subsubsection{Adjoint System}
Consider the following homogeneous initial-value problem associated to (\ref{gg1})-(\ref{gg2}):
\begin{equation}\label{gglin3}
\left\lbrace \begin{tabular}{l l}
$u_t + u_{xxx} + av_{xxx}  =0$, & in $(0,L)\times (0,T)$,\\
$v_t +\frac{r}{c}v_x+\frac{ab}{c}u_{xxx} + \frac{1}{c}v_{xxx} =0$, & in $(0,L)\times (0,T)$,\\
$u_{xx}(0,t)=u_x(L,t)=u_{xx}(L,t) = 0$,& in $(0,T)$,\\
$v_{xx}(0,t) =v_x(L,t) =  v_{xx}(L,t) = 0$,& in $(0,T)$,\\
$u(x,0)=u^0(x), \quad v(x,0) = v^0(x)$, & in $(0,L)$.
\end{tabular}\right.
\end{equation}
 In order to introduce the backward system associated to \eqref{gglin3}, we multiply the first equation of \eqref{gglin3} by $\varphi$, the second one by $\psi$ and integrate over $(0,L)\times (0,T)$. Assuming that the functions $u, v, \varphi$ and $\psi$ are regular enough to justify all the computations, we obtain, after integration by parts, the following identity:
\begin{align*}\label{ad1}
 \int_0^L&\left( u(x,T)\varphi(x,T)+v(x,T)\psi(x,T)\right)dx-\int_0^L\left(u^0(x)\varphi(x,0)+v^0(x)\psi(x,0)\right) dx =  \notag \\
 &\int_0^T\int_0^L u(x,t)\left(\varphi(x,t) + \varphi_{xxx}(x,t) + \frac{ab}{c}\psi_{xxx}(x,t) \right)dxdt \notag \\
&+\int_0^T\int_0^Lv(x,t)\left( \psi(x,t) +\frac{r}{c}\psi(x,t)+a\varphi_{xxx}(x,t) + \frac{1}{c}\psi_{xxx}(x,t)\right)dxdt  \notag\\
&-\int_0^T u_{xx}(L,t)\left(\varphi(L,t)+\frac{ab}{c}\psi(L,t)\right)dt+\int_0^T u_{xx}(0,t)\left(\varphi(0,t)+\frac{ab}{c}\psi(0,t)\right)dt \notag\\
&+\int_0^T u_x(L,t)\left(\varphi_x(L,t)+\frac{ab}{c}\psi_x(L,t)\right)dt-\int_0^T u_x(0,t)\left(\varphi_x(0,t)+\frac{ab}{c}\psi_x(0,t)\right)dt \notag \\
&-\int_0^T u(L,t)\left(\varphi_{xx}(L,t)+\frac{ab}{c}\psi_{xx}(L,t)\right)dt+\int_0^T u(0,t)\left(\varphi_{xx}(0,t)+\frac{ab}{c}\psi_{xx}(0,t)\right)dt  \\
&-\int_0^T v_{xx}(L,t)\left(a\varphi(L,t)+\frac{1}{c}\psi(L,t)\right)dt+\int_0^T v_{xx}(0,t)\left(a\varphi(0,t)+\frac{1}{c}\psi(0,t)\right)dt \notag \\
&+\int_0^T v_{x}(L,t)\left(a\varphi_x(L,t)+\frac{1}{c}\psi_x(L,t)\right)dt-\int_0^T v_{x}(0,t)\left(a\varphi_x(0,t)+\frac{1}{c}\psi_x(0,t)\right)dt \notag \\
&-\int_0^T v(L,t)\left(a\varphi_{xx}(L,t)+\frac{1}{c}\psi_{xx}(L,t)+\frac{r}{c}\psi(L,t)\right)dt\notag \\
&+\int_0^T v(0,t)\left(a\varphi_{xx}(0,t)+\frac{1}{c}\psi_{xx}(0,t)+\frac{1}{c}\psi(0,t)\right)dt. \notag
\end{align*}
Having the previous equality in hands, we consider backward system as follows
\begin{equation}\label{linadj}
\begin{cases}
\varphi_t + \varphi_{xxx} + \frac{ab}{c}\psi_{xxx}=0,  & \text{in}\,\, (0,L)\times (0,T), \\
\psi_t  +\frac{r}{c}\psi_x+a\varphi_{xxx} +\frac{1}{c}\psi_{xxx} =0,  & \text{in}\,\, (0,L)\times (0,T)
\end{cases}
\end{equation}
satisfying the boundary conditions,
\begin{equation}\label{bc1}
\begin{cases}
a\varphi_x(0,t)+\frac{1}{c}\psi_x(0,t) =0, & \text{in}\,\, (0,T),\\
\varphi_x (0,t) +\frac{ab}{c}\psi_x (0,t) =0, & \text{in}\,\, (0,T), \\
\varphi_{xx} (L,t) +\frac{ab}{c}\psi_{xx} (L,t) =0, & \text{in}\,\,  (0,T), \\
\varphi_{xx} (0,t) +\frac{ab}{c}\psi_{xx} (0,t) =0, & \text{in}\,\,  (0,T), \\
a\varphi_{xx}(L,t)+\frac{1}{c}\psi_{xx}(L,t)+\frac{r}{c}\psi(L,t) =0, & \text{in}\,\,  (0,T), \\
a\varphi_{xx}(0,t)+\frac{1}{c}\psi_{xx}(0,t)+\frac{r}{c}\psi(0,t) =0, & \text{in}\,\, (0,T)
\end{cases}
\end{equation}
and the final conditions
\begin{equation}\label{finaladj}
\varphi(x,T)= \varphi^1(x), \qquad \psi(x,T)= \psi^1(x),  \qquad  \text{in}\,\, (0,L).
\end{equation}
Since the coefficients satisfy $1-a^2b >0$, we can deduce from the first and second equations of \eqref{bc1} that the above boundary conditions can be written as
\begin{equation}\label{linadjbound}
\begin{cases}
\varphi_x(0,t)=\psi_x(0,t) =0, & \text{in}\,\, (0,T),\\
\varphi_{xx} (L,t) +\frac{ab}{c}\psi_{xx} (L,t) =0, & \text{in}\,\, (0,T),\\
\varphi_{xx} (0,t) +\frac{ab}{c}\psi_{xx} (0,t) =0, & \text{in}\,\, (0,T),\\
a\varphi_{xx}(L,t)+\frac{1}{c}\psi_{xx}(L,t)+\frac{r}{c}\psi(L,t) =0, & \text{in}\,\, (0,T), \\
a\varphi_{xx}(0,t)+\frac{1}{c}\psi_{xx}(0,t)+\frac{r}{c}\psi(0,t) =0, & \text{in}\,\, (0,T).
\end{cases}
\end{equation}

The following proposition is the key to prove the controllability of the linear system \eqref{gglin}. The result ensures the hidden regularity for the solution of the adjoint system \eqref{linadj}-\eqref{linadjbound}.
\begin{prop}\label{hiddenregularities}
For any $(\varphi^1, \psi^1) \in \X$, the system \eqref{linadj}-\eqref{linadjbound} admits a unique solution $(\varphi, \psi) \in \ZZ_T$, such that it possess the following sharp trace properties
\begin{equation}\label{hr4}
\begin{cases}
\underset{0< x < L}{\sup} \|  \partial^k_x \varphi(x,\cdot)\|_{H^{\frac{1-k}3}(0,T)}\le C_T\|\varphi^1\|_{L^2(0,L)}, \\
 \underset{0< x <L}{\sup} \|  \partial^k_x \psi(x,\cdot)\|_{H^{\frac{1-k}3}(0,T)}\le C_T \|\psi^1\|_{L^2(0,L)},
\end{cases}
  \end{equation}
for $k=0,1,2$, where $C_r$ is a positive constant.
\end{prop}
\begin{proof}
Proceeding as the proof of Proposition \ref{prop2}, we obtain the result. Indeed, first we consider the change of variable $t\rightarrow T-t$ and $x\rightarrow L-x$, then for any $(\varphi,\psi)$ in $\ZZ_T$, we consider the system
\begin{equation*}
\begin{cases}
u_t + u_{xxx} + \frac{ab}{c}v_{xxx}=0,  & \text{in}\,\, (0,L)\times (0,T), \\
v_t  + au_{xxx} +\frac{1}{c}v_{xxx} =-\frac{r}{c}v_x,  & \text{in}\,\, (0,L)\times (0,T), \\
\varphi(x,0)= \varphi^0(x)\text{, } \psi(x,0)= \psi^0(x),  &  \text{in}\,\, (0,L),
\end{cases}
\end{equation*}
with boundary conditions
\begin{equation*}
\begin{cases}
u_x(L,t)=v_x(L,t) =0, & \text{in}\,\, (0,T),\\
u_{xx} (L,t)=-\frac{ab}{c}\psi_{xx} (L,t), & \text{in}\,\, (0,T),\\
u_{xx} (0,t)=-\frac{ab}{c}\psi_{xx} (0,t), & \text{in}\,\, (0,T),\\
v_{xx}(L,t)=-ac\varphi_{xx}(L,t)-r\psi(L,t), & \text{in}\,\, (0,T), \\
v_{xx}(0,t)=-ac\varphi_{xx}(0,t)-r\psi(0,t), & \text{in}\,\, (0,T).
\end{cases}
\end{equation*}
By using a fixed point argument the result is archived.
\end{proof}

The adjoint system possesses a relevant  estimate as described below.
\begin{prop}\label{prop3}
Any solution $(\varphi,\psi)$ of the adjoint system  \eqref{linadj}-\eqref{linadjbound} satisfies
\begin{align}\label{e6}
\|(\varphi^1,\psi^1)\|_{\X}^2 \leq &\frac{1}{T}\|(\varphi,\psi)\|_{L^2(0,T;\X)}+\frac{1}{2}\|\varphi_x(L,\cdot)\|_{L^2(0,T)}^2+\frac{b}{2c}\|\psi_x(L,\cdot)\|_{L^2(0,T)}^2 +\frac{br}{c^2}\|\psi(L,\cdot)\|_{L^2(0,T)}^2\notag\\
&+\frac{1}{2}\left\|\varphi_x(L,\cdot)+\frac{ab}{c}\psi_x(L,\cdot)\right\|_{L^2(0,T)}^2+\frac{b}{2c}\left\|a\varphi_x(L,\cdot)+\frac{1}{c}\psi_x(L,\cdot)\right\|_{L^2(0,T)}^2,
\end{align}
with initial data $(\varphi^1,\psi^1) \in \X$.
\end{prop}

\begin{proof}
Multiplying the first equation of \eqref{linadj} by $-t\varphi$, the second one by $-\frac{b}{c}t\psi$ and integrating by parts over $(0, T)\times (0,L)$, we obtain
\begin{multline*}
\frac{T}{2}\int_0^L\varphi^2(x,T)dx=\frac{1}{2}\int_0^T\int_0^L\varphi^2(x,t)dxdt + \frac{ab}{c}\int_0^T\int_0^L t\varphi_{xxx}(x,t)\psi(x,t) dxdt\\
-\int_0^Tt\left[\varphi_{xx}(x,t)\varphi(x,t)-\frac{1}{2}\varphi_x^2(x,t)+\frac{ab}{c}\psi_{xx}(x,t)\varphi(x,t)-\frac{ab}{c}\psi_{x}(x,t)\varphi_x(x,t)\right. \\
\left.+\frac{ab}{c}\psi(x,t)\varphi_{xx}(x,t)\right]_0^Ldt
\end{multline*}
and
\begin{multline*}
\frac{Tb}{2c}\int_0^L\psi^2(x,T)dx=\frac{b}{2c}\int_0^T\int_0^L\psi^2(x,t)dxdt - \frac{ab}{c}\int_0^T\int_0^L t\varphi_{xxx}(x,t)\psi(x,t) dxdt \\
-\int_0^Tt\left[\frac{b}{c^2}\psi_{xx}(x,t)\psi(x,t)-\frac{b}{2c^2}\psi_x^2(x,t)+\frac{br}{2c^2}\psi^2(x,t)\right]_0^Ldt.
\end{multline*}
Adding the above identities, it follows that
\begin{multline*}
\frac{T}{2}\|(\varphi^1,\psi^1)\|_{\X}^2 = \frac{1}{2}\|(\varphi,\psi)\|^2_{L^2(0,T;\X)} - \int_0^Tt\left[\frac{b}{c}\psi(x,t)\left(a\varphi_{xx}(x,t)+\frac1c\psi_{xx}(x,t)+\frac{r}{c}\psi(x,t)\right) \right]_0^Ldt \\
-\int_0^Tt\left[\frac{b}{2c}\psi_x(x,t)\left(a\varphi_{x}(x,t)+\frac1c\psi_{x}(x,t)\right) -\frac{1}{2}\varphi_x(x,t)\left(\varphi_{x}(x,t)+\frac{ab}{c}\psi_{x}(x,t)\right)\right]_0^Ldt \\
+\int_0^Tt\left[\varphi(x,t)\left(\varphi_{xx}(x,t)+\frac{ab}{c}\psi_{xx}(x,t)\right) -\frac{br}{2c^2}\psi^2(x,t)\right]_0^Ldt.
\end{multline*}
Then, from \eqref{linadjbound}, we obtain
\begin{align*}
\frac{T}{2}\|(\varphi^1,\psi^1)\|_{\X}^2 \leq & \frac{1}{2}\|(\varphi,\psi)\|^2_{L^2(0,T;\X)} + \frac{b T}{2c}\int_0^T\psi_x(L,t)\left(a\varphi_{x}(L,t)+\frac1c\psi_{x}(L,t)\right)dt \\
&+\frac{T}{2}\int_0^T\varphi_x(L,t)\left(\varphi_{x}(L,t)+\frac{ab}{c}\psi_{x,t}(L,t)\right)dt \\
&+\frac{brT}{2c^2}\int_0^T\psi^2(L,t)dt-\frac{brT}{2c^2}\int_0^T\psi^2(0,t)dt.
\end{align*}
Finally, \eqref{e6} is obtained by applying Young inequality in the right hand side of the above inequality.
\end{proof}

\subsection{Nonlinear System}

In this subsection, attention will be given to the full nonlinear system \eqref{gg1}-\eqref{gg2}. The proof of the lemma below is available in \cite[Lemma 3.1]{bonasunzhang2003} and, therefore, we will omit it.
\begin{lemma}\label{lem3}
There exists a constant $C>0$, such that, for any $T>0$ and $(u,v) \in \ZZ_T$,
\begin{equation*}
\|uv_x\|_{L^1(0,T;L^2(0,L))}\leq C(T^{\frac{1}{2}}+T^\frac{1}{3})\|u\|_{\ZZ_T}\|v\|_{\ZZ_T}.
\end{equation*}
\end{lemma}

We first show that system \eqref{gg1}-\eqref{gg2} is locally well-posed in the space $\ZZ_T$.

\begin{thm}\label{nonlinearteo}
For any $(u^0,v^0) \in \X$ and $\overrightarrow{h}=(h_0,h_1,h_2), \overrightarrow{g}=(g_0,g_1,g_2) \in \HH_T$, there exists $T^*>0$, depending on $\|(u^0,v^0)\|_{\X}$, such that the problem \eqref{gg1}-\eqref{gg2} admits a unique solution $(u,v) \in \ZZ_{T^*}$ with
\begin{equation*}
\partial_x^k u,\partial_x^k v \in L^{\infty}_x(0,L;H^{\frac{1-k}{3}}(0,T^*)),  \quad k=0,1,2.
\end{equation*}
Moreover, the corresponding solution map is Lipschitz continuous.
\end{thm}

\begin{proof}
Let $\FF_T= \left\lbrace (u,v) \in \ZZ_T: (u,v) \in  L^{\infty}_x(0,L;(H^{\frac{1-k}{3}}(0,T))^2), k=0,1,2\right\rbrace$ equipped with the norm
\[
\|(u,v)\|_{\FF_T} = \|(u,v)\|_{\ZZ_T}+\sum_{k=0}^{2}\|(u,v)\|_{L^{\infty}_x(0,L;(H^{\frac{1-k}{3}}(0,T))^2)}.
\]
Let $0< T^* \leq T$ to be determined later. For each $u,v \in \FF_{T^*}$, consider the problem
\begin{equation}\label{e1'}
\left\lbrace \begin{tabular}{l l}
$\omega_t + \omega_{xxx} + a\eta_{xxx}  =f(u,v)$, & in $(0,L)\times (0,T^*)$,\\
$\eta_t +\frac{ab}{c}\omega_{xxx} + \frac{1}{c}\eta_{xxx} =s(u,v)$, & in $(0,L)\times (0,T^*)$,\\
$\omega_{xx}(0,t) = h_0(t),\,\,\omega_x(L,t) = h_1(t),\,\,\omega_{xx}(L,t) = h_2(t)$,& in $(0,T^*)$,\\
$\eta_{xx}(0,t) =g_0(t),\,\,\eta_x(L,t) = g_1(t),\,\, \eta_{xx}(L,t) = g_2(t)$,& in $(0,T^*)$,\\
$\omega(x,0)=u^0(x), \quad v(x,0) = v^0(x)$, & in $(0,L)$,
\end{tabular}\right.
\end{equation}
where $$f(u,v)=-a_1(vv_x)-a_2(uv)_x$$ and $$s(u,v)=-\frac{r}{c}v_x -\frac{a_2b}{c}(uu_x)-\frac{a_1b}{c}(uv)_x.$$
Since $\|v_x\|_{L^1(0,T^*;L^2(0,L))}\leq \beta^{\frac12}\|v\|_{\ZZ_{T^*}}$, from Lemma \ref{lem3} we deduce that $f(u,v)$ and $s(u,v)$ belong to $L^1(0,T^*;L^2(0,L))$ and satisfies
\begin{equation}\label{f,g}
\|(f,s)\|_{L^1(0,T^*;(L^2(0,L))^2)}\leq C_1 ((T^*)^{\frac12}+(T^*)^{\frac13})\left( \|u\|_{\ZZ_{T^*}}^2+(\|u\|_{\ZZ_{T^*}}+1)\|v\|_{\ZZ_{T^*}}+\|v\|_{\ZZ_{T^*}}^2\right),
\end{equation}
for some positive constant $C_1$. Then, according to Proposition \ref{prop1}, we can define the operator
\[
\Gamma: \FF_{T^*} \rightarrow \FF_{T^*}, \quad \text{given by} \quad \Gamma(u,v)=(\omega,\eta),
\]
where $(\omega,\eta)$ is the solution of (\ref{e1'}). Moreover,
\begin{equation}\label{f,g,h}
\|\Gamma(u,v)\|_{\FF_{T^*}}\leq C\left\lbrace \|(u^0,v^0)\|_{\X}+\|(\overrightarrow{h},\overrightarrow{g})\|_{\HH_{T^*}}+\|(f,s)\|_{L^1(0,T^*;(L^2(0,L))^2)} \right\rbrace,
\end{equation}
where the positive  constant $C$ depends only on $T^*$. Combining \eqref{f,g} and \eqref{f,g,h}, we obtain
\begin{align*}
\|\Gamma(u,v)\|_{\FF_{T^*}}\leq &C\left\lbrace \|(u^0,v^0)\|_{\X}+\|(\overrightarrow{h},\overrightarrow{g})\|_{\HH_{T^*}}\right\rbrace \\
&+CC_1((T^*)^{\frac12}+(T^*)^{\frac13})\left( \|u\|_{\ZZ_{T^*}}^2+(\|u\|_{\ZZ_{T^*}}+1)\|v\|_{\ZZ_{T^*}}+\|v\|_{\ZZ_{T^*}}^2\right).
\end{align*}
Let $(u,v) \in B_{r}(0):=\left\lbrace (u,v) \in \FF_{T^*}: \|(u,v)\|_{\FF_{T^*}}\leq r\right\rbrace$, where $r=2C\left\lbrace \|(u^0,v^0)\|_{\X}+\|(\overrightarrow{h},\overrightarrow{g})\|_{\HH_T}\right\rbrace$.
From the estimate above, it follows that
\begin{equation}\label{e2'}
\|\Gamma(u,v)\|_{\FF_{T^*}}\leq \frac{r}{2}+CC_1((T^*)^{\frac12}+(T^*)^{\frac13})\left( 3r+1\right)r.
\end{equation}
Then, by choosing $T^{*}>0$, such that
\begin{equation}\label{beta1}
CC_1((T^*)^{\frac12}+(T^*)^{\frac13})\left( 3r+1\right)\leq \frac{1}{2},
\end{equation}
from (\ref{e2'}), we have
$$\|\Gamma(u,v)\|_{\FF_{T^*}}\leq r.$$
Thus, we conclude that
\[
\Gamma: B_r(0)\subset\FF_{T^*} \rightarrow B_r(0).
\]
On the other hand,  $\Gamma(u_1,v_1)-\Gamma(u_2,v_2)$ solves the system
\begin{equation*}
\left\lbrace \begin{tabular}{l l}
$\omega_t + \omega_{xxx} + a\eta_{xxx}  =f(u_1,v_1)-f(u_2,v_2)$, & in $(0,L)\times (0,T^*)$,\\
$\eta_t +\frac{ab}{c}\omega_{xxx} + \frac{1}{c}\eta_{xxx} =s(u_1,v_1)-s(u_2,v_2) $, & in $(0,L)\times (0,T^*)$,\\
$\omega_{xx}(0,t) = \omega_x(L,t) = \omega_{xx}(L,t) = 0$,& in $(0,T^*)$,\\
$\eta_{xx}(0,t)=\eta_x(L,t) = \eta_{xx}(L,t) = 0$,& in $(0,T^*)$,\\
$\omega(x,0)=0, \quad v(x,0) = 0$, & in $(0,L)$,
\end{tabular}\right.
\end{equation*}
where, $f(u,v)$ and $s(u,v)$ were defined in \eqref{e1'}. Note that
\begin{align*}
|f(u_1,v_1)-f(u_2,v_2)| \leq  C_2|\left( (v_2-v_1)v_{2,x}+ v_1(v_2-v_1)_x+(u_2(v_2-v_1))_x +((u_2-u_1)v_1)_x\right)|
\end{align*}
and
\begin{multline*}
|s(u_1,v_1)-s(u_2,v_2)|  \leq C_2|\left( (v_2-v_1)_x+ (u_2-u_1)u_{2,x}+ u_1(u_2-u_1)_x \right. \\
\left. +(u_2(v_2-v_1))_x +((u_2-u_1)v_1)_x\right)|,
\end{multline*}
for some positive constant $C_2$. Then, Proposition \ref{prop1} and Lemma \ref{lem3}, give us the following estimate
\begin{align*}
\|\Gamma(u_1,v_1)-\Gamma(u_2,v_2)\|_{\FF_{T^*}}&\leq C\|(f(u_1,v_1)-f(u_2,v_2), s(u_1,v_1)-s(u_2,v_2))\|_{L^1(0,T^*;(L^2(0,L))^2)} \\
&\leq C_3((T^*)^{\frac{1}{2}}+(T^*)^{\frac{1}{3}})(8r+1)\|(u_1-u_2,v_1-v_2)\|_{\FF_{T^*}},
\end{align*}
for some positive constant $C_3$. Choosing $T^*$, satisfying \eqref{beta1} and such that
$$C_3((T^*)^{\frac{1}{2}}+(T^*)^{\frac{1}{3}})(8r+1) \leq \frac{1}{2},$$
we obtain
\[\|\Gamma(u_1,v_1)-\Gamma(u_2,v_2)\|_{\FF_{T^*}}\leq \frac{1}{2}\|(u_1-u_2,v_1-v_2)\|_{\FF_{T^*}}.
\]
Hence, $\Gamma: B_r(0) \rightarrow B_r(0)$ is a contraction and, by Banach fixed point theorem, we obtain a unique $(u,v) \in B_r(0)$, such that $\Gamma(u,v) = (u,v) \in \FF_{T^*}$ and, therefore, the proof is complete.
\end{proof}
\begin{remark}\label{integralform}
From the proof of Proposition \ref{prop1},  we deduce that solution of the system (\ref{gg1})-\eqref{gg2} can be written as
\begin{align*}
\left( \begin{array}{cc}
u(t)\\ v(t)
\end{array} \right)&= W_0(t)\left( \begin{array}{cc}
u^0(x) \\ v^0(x)
\end{array} \right) + W_{bdr}(t)\left( \begin{array}{cc}
\overrightarrow{h} \\ \overrightarrow{g}
\end{array} \right)\\
& -\int_0^t W_0(t-\tau) \left( \begin{array}{cc}
a_1(vv_x)(\tau)+a_2(uv)_x(\tau)\vspace{2mm} \\
\frac{r}{c}v_x(\tau)+\frac{a_2b}{c}(uu_x)(\tau)+\frac{a_1b}{c}(uv)_x(\tau)
\end{array} \right)d\tau,
\end{align*}
with
\begin{equation*}
W_0(t) = \left( \begin{array}{cc}
W_0^{-}(t) & 0 \\
0 & W_0^{+}(t)
\end{array}\right) \quad \text{and} \quad W_{bdr}(t) = \left( \begin{array}{cc}
W_{bdr}^{-}(t) & 0 \\
0 & W_{bdr}^{+}(t)
\end{array}\right),
\end{equation*}
where $\{W_0^{\pm}(t)\}_{t\geq 0}$ are the $C_0$-semigroup in the space $L^2(0,L)$ generated by the linear operators
\[
A^{\pm}=-\alpha_{\pm}g''',
\]
where
\[
\alpha_{\pm} = -\frac{1}{2}\left(\left(\frac{1}{c}-1\right)\pm \sqrt{\left(\frac{1}{c}-1\right)^2+\frac{4a^2b}{c}}\right),
\]
with domain $$D(A^{\pm})=\{g \in H^3(0,L): g''(0)=g'(L)=g''(L)=0\},$$ and $\{W_{bdr}^{\pm}(x)\}_{t\geq 0}$ is the operator defined in Lemma\ref{l4-d}.
\end{remark}

\section{Exact Boundary Controllability for the Linear System}\label{Sec3}

In this section, we study the existence of controls $\overrightarrow{h}:=(h_0,h_1,h_2)$ and $\overrightarrow{g}:=(g_0,g_1,g_2) \in \HH_T$,
such that the solution $(u,v)$ of the system
\begin{equation}\label{ggln4}
\left\lbrace \begin{tabular}{l l}
$u_t + u_{xxx} + av_{xxx}  =0$ & in $(0,L)\times (0,T)$,\\
$v_t +\frac{r}{c}v_x+\frac{ab}{c}u_{xxx} + \frac{1}{c}v_{xxx} =0$ & in $(0,L)\times (0,T)$,\\
$u(x,0)=u^0(x), \quad v(x,0) = v^0(x)$, & in $(0,L)$,
\end{tabular}\right.
\end{equation}
satisfying the boundary conditions
\begin{equation}\label{ggln4b}
\left\lbrace\begin{tabular}{l l}
$u_{xx}(0,t) = h_0(t),\,\,u_x(L,t) = h_1(t),\,\,u_{xx}(L,t) = h_2(t)$ & in $(0,T)$,\\
$v_{xx}(0,t) =g_0(t),\,\,v_x(L,t) = g_1(t),\,\, v_{xx}(L,t) = g_2(t)$ & in $(0,T)$,
\end{tabular}\right.
\end{equation}
satisfies
\begin{equation}\label{finaldata}
u(\cdot,T)=u^1(\cdot), \qquad   \text{and}  \qquad v(\cdot,T)=v^1(\cdot).
\end{equation}

More precisely, we have the following definition:

\begin{definition}
Let $T > 0$. System \eqref{ggln4}-\eqref{ggln4b} is exact controllable in time $T$ if for any initial and final data $(u^0,v^0)$ and $(u^1,v^1)$ in $\X$, there exist control functions $\overrightarrow{h}=(h_0,h_1,h_2)$ and $\overrightarrow{g}=(g_0,g_1,g_2)$ in $\HH_T$,  such that the solution of \eqref{ggln4}-\eqref{ggln4b} satisfies \eqref{finaldata}.
\end{definition}
\begin{remark}
Without any loss of generality, we shall consider only the case $u^0 = v^0 = 0$. Indeed, let $(u^0,v^0)$, $(u^1,v^1)$ in $\X$ and $\overrightarrow{h}$, $\overrightarrow{g}$ in $\HH_T$ be controls which lead the solution $(\widetilde{u}, \widetilde{v})$ of \eqref{ggln4} from the zero initial data to the final state $(u^1,v^1)-(u(T),v(T))$, where $(u,v)$ is the mild solution corresponding to \eqref{ggln4}-\eqref{ggln4b} with initial data $(u^0,v^0)$. It follows immediately that these controls also lead to the solution $(\widetilde{u},\widetilde{v})+(u,v)$ of \eqref{ggln4}-\eqref{ggln4b} from $(u^0,v^0)$ to the final state $(u^1,v^1)$.
\end{remark}

In the following pages, we will analyze the exact controllability of the system \eqref{ggln4}-\eqref{ggln4b} for different combinations of four controls and one control.

\subsection{Four Controls}
\subsubsection{Case 1} Consider the following boundary conditions:
\begin{equation}\label{gglnb1}
\left\lbrace\begin{tabular}{l l}
$u_{xx}(0,t) = h_0(t),\,\,u_x(L,t) = h_1(t),\,\,u_{xx}(L,t) = 0$ & in $(0,T)$,\\
$v_{xx}(0,t) =g_0(t),\,\,v_x(L,t) = g_1(t),\,\, v_{xx}(L,t) = 0$ & in $(0,T)$.
\end{tabular}\right.
\end{equation}
We first give an equivalent condition for the exact controllability property.
\begin{lemma}\label{equicontrol}
For any $(u^1,v^1)$ in $\X$, there exist four controls $\overrightarrow{h}=(h_0,h_1,0)$ and $\overrightarrow{g}=(g_0,g_1,0)$ in $\HH_T$, such that the solution $(u,v)$ of \eqref{ggln4}-\eqref{gglnb1} satisfies \eqref{finaldata} if and only if
\begin{align}
\int_0^L\left(u^1(x)\varphi^1(x)+v^1(x)\psi^1(x)\right) dx =& \int_0^T h_0(t)\left(\varphi(0,t)+\frac{ab}{c}\psi(0,t)\right)dt\nonumber\\
& +\int_0^T h_1(t)\left(\varphi_x(L,t)+\frac{ab}{c}\psi_x(L,t)\right)dt\nonumber \\
&+\int_0^T g_0(t)\left(a\varphi(0,t)+\frac{1}{c}\psi(0,t)\right)dt\label{ad2}\\
&+\int_0^T g_1(t)\left(a\varphi_x(L,t)+\frac{1}{c}\psi_x(L,t)\right)dt,\nonumber
\end{align}
for any $(\varphi^1,\psi^1)$ in $\X$, where $(\varphi,\psi)$ is the solution of the backward system \eqref{linadj}-\eqref{linadjbound} with initial data $(\varphi^1,\psi^1)$.
\end{lemma}
\begin{proof}
The relation \eqref{ad2} is obtained by multiplying the equations in \eqref{ggln4} by the solution $(\varphi,\psi)$ of  \eqref{linadj}-\eqref{linadjbound}, integrating by parts and using the boundary conditions \eqref{gglnb1}.
\end{proof}

The following observability inequality plays a fundamental role for the study of the controllability properties.
\begin{prop}\label{prop4}
For $T>0$ and $L> 0$, there exists a constant $C:=C(T,L) > 0$, such that
\begin{align}\label{obineq1}
\|(\varphi^1, \psi^1)\|_{\X}^2\leq C&\left\lbrace  \left\|(-\Delta_t)^{\frac16}\left(\varphi(0,\cdot)+\frac{ab}{c}\psi(0,\cdot)\right)\right\|_{L^2(0,T)}^2 + \left \|\varphi_x(L,\cdot)+\frac{ab}{c}\psi_x(L,\cdot)\right \|_{L^2(0,T)}^2 \right. \notag \\
&\left. +\left \|(-\Delta_t)^{\frac16}\left(a\varphi(0,\cdot)+\frac{1}{c}\psi(0,\cdot)\right)\right \|_{L^2(0,T)}^2 + \left \|a\varphi_x(L,\cdot)+\frac{1}{c}\psi_x(L,\cdot)\right \|_{L^2(0,T)}^2\right\rbrace,
\end{align}
for any $(\varphi^1,\psi^1) \in \X$, where $(\varphi,\psi)$ is a solution of \eqref{linadj}-\eqref{linadjbound} with initial data $(\varphi^1,\psi^1)$, where $\Delta_t:=\partial_t^2$.
\end{prop}

\begin{proof}
We argue by contradiction, as in \cite[Proposition 3.3]{rosier}, and suppose that \eqref{obineq1} does not hold. In this case, we obtain a sequence $\{(\varphi^1_n, \psi^1_n)\}_{n \in \N}$, satisfying
\begin{align}\label{e3}
1=\|(\varphi^1_n, \psi^1_n)\|_{\X}^2&\geq n\left\lbrace \left \|(-\Delta_t)^{\frac16}\left(\varphi_n(0,\cdot)+\frac{ab}{c}\psi_n(0,\cdot)\right)\right \|_{L^{2}(0,T)}^2 + \left  \|\varphi_{n,x}(L,\cdot)+\frac{ab}{c}\psi_{n,x}(L,\cdot)\right \|_{L^2(0,T)}^2 \right. \notag \\
&\left. +\left \|(-\Delta_t)^{\frac16}\left(a\varphi_n(0,\cdot)+\frac{1}{c}\psi_n(0,\cdot)\right)\right \|_{L^2(0,L)}^2 + \left \|a\varphi_{n,x}(L,\cdot)+\frac{1}{c}\psi_{n,x}(L,\cdot)\right \|_{L^2(0,T)}^2\right\rbrace.
\end{align}
Consequently, \eqref{e3} imply that
\begin{equation}\label{e4}
\begin{cases}
\varphi_n(0,\cdot)+\frac{ab}{c}\psi_n(0,\cdot) \rightarrow 0 & \text{in} \quad H^{\frac{1}{3}}(0,T), \\
a\varphi_n(0,\cdot)+\frac{1}{c}\psi_n(0,\cdot) \rightarrow 0 & \text{in} \quad H^{\frac{1}{3}}(0,T), \\
\varphi_{n,x}(L,\cdot)+\frac{ab}{c}\psi_{n,x}(L,\cdot) \rightarrow 0 & \text{in} \quad L^2(0,T), \\
a\varphi_{n,x}(L,\cdot)+\frac{1}{c}\psi_{n,x}(L,\cdot)  \rightarrow 0 & \text{in} \quad L^2(0,T),
\end{cases}
\end{equation}
as $n\rightarrow\infty$. Since $1-a^2b >0$, \eqref{e4} guarantees that the following converge hold
\begin{equation}\label{e4'}
\begin{cases}
\varphi_n(0,\cdot) \rightarrow 0, \quad \psi_n(0,\cdot) \rightarrow 0 &\text{in} \quad H^{\frac{1}{3}}(0,T),\\
\varphi_{n,x}(L,\cdot)\rightarrow 0,\quad \psi_{n,x}(L,\cdot) \rightarrow 0 &\text{in} \quad L^2(0,T),\\
\end{cases}
\end{equation}
as $n\rightarrow\infty$. The next steps are devoted to pass the strong limit in the left hand side of \eqref{e3}. First, observe that from Proposition \ref{hiddenregularities} we deduce that
$\{ (\varphi_n,\psi_n)\}_{n \in \N}$ is bounded in $L^2(0,T;(H^1(0,L))^2)$. Then, \eqref{linadj} implies that $\{ (\varphi_{t,n},\psi_{t,n})\}_{n \in \N}$ is bounded in $L^2(0,T; (H^{-2}(0,L))^2)$ and the compact embedding
\begin{equation*}
 H^1(0,L) \hookrightarrow L^2(0,L) \hookrightarrow H^{-2}(0,L)
\end{equation*}
 allows us to  conclude that $\{ (\varphi_n,\psi_n)\}_{n \in \N}$ is relatively compact in $L^2(0,T;\X)$. Consequently, we obtain a subsequence, still denoted by the same index $n$, satisfying
 \begin{equation}\label{e8}
(\varphi_n,\psi_n) \rightarrow (\varphi,\psi) \mbox{ in } L^2(0,T;\X), \mbox{ as } n\rightarrow\infty.
 \end{equation}
On the other hand, \eqref{hr4} and \eqref{e3} imply that the sequences
\[
\text{$\{ \varphi_n(0,\cdot)\}_{n\in\N}$\mbox{ and } $\{\psi_n(0,\cdot)\}_{n\in\N}$ are bounded in $H^{\frac13}(0,T)$.}
\]
Then, the following compact embedding
\begin{equation}\label{e15}
H^{\frac13}(0,T) \hookrightarrow L^2(0,T)
\end{equation}
guarantees that the above sequences are relatively compact in $L^2(0,T)$, that is, we obtain a subsequence, still denoted by the same index $n$, satisfying
\begin{equation}\label{e10}
\begin{cases}
\varphi_n(0,\cdot) \rightarrow \varphi(0,\cdot) \quad \text{in} \quad L^2(0,T), \\
\psi_n(0,\cdot) \rightarrow \psi(0,\cdot) \quad \text{in} \quad L^2(0,T),
\end{cases}
\end{equation}
as $n\rightarrow\infty$. Then, from \eqref{e4'} and \eqref{e10} we deduce that
\begin{equation*}
\varphi(0,\cdot)=\psi(0,\cdot)=0.
\end{equation*}
Moreover, \eqref{hr4}, \eqref{e3} and \eqref{e15} imply that $\{\varphi_n(L,t)\}_{n\in\N}$ and $\{\psi_n(L,t)\}_{n\in\N}$ are relatively compact in $L^2(0,T)$. Hence, we obtain a subsequence, still denoted by the same index, satisfying
\begin{equation}\label{e9}
 \begin{cases}
  \varphi_n(L,\cdot) \rightarrow \varphi(L,\cdot) \mbox{ in } L^2(0,T), \\
 \psi_n(L,\cdot) \rightarrow \psi(L,\cdot) \mbox{ in } L^2(0,T),
 \end{cases}
 \end{equation}
as $n\rightarrow\infty$. In addition, according to Proposition \ref{prop3}, we have
\begin{align*}
\|(\varphi^1_n,\psi^1_n)\|_{\X}^2 \leq& \frac{1}{T}\|(\varphi_n,\psi_n)\|_{L^2(0,T;\X)}+\frac{1}{2}\|\varphi_{n,x}(L,\cdot)\|_{L^2(0,T)}^2\\
&+\frac{b}{2c}\|\psi_{n,x}(L,\cdot)\|_{L^2(0,T)}^2  +\frac{br}{c^2}\|\psi_{n}(L,\cdot)\|_{L^2(0,T)}^2\\
&+\frac{1}{2}\left\|\varphi_{n,x}(L,\cdot)+\frac{ab}{c}\psi_{n,x}(L,\cdot)\right\|_{L^2(0,T)}^2+\frac{b}{2c}\left\|a\varphi_{n,x}(L,\cdot)+\frac{1}{c}\psi_{n,x}(L,\cdot)\right\|_{L^2(0,T)}^2.
\end{align*}
Then, from \eqref{e4}, \eqref{e4'}, \eqref{e8}  and \eqref{e9} we conclude that $\{(\varphi^1_n,\psi^1_n)\}_{n\in\N}$ is a Cauchy sequence in $\X$ and, therefore, we get
\begin{equation}\label{e16}
(\varphi^1_n,\psi^1_n) \rightarrow (\varphi^1,\psi^1) \mbox{ in } \X, \mbox{ as } n\rightarrow\infty.
\end{equation}
Thus, Proposition \ref{hiddenregularities} together with \eqref{e16} imply that
\begin{equation}\label{eeee1}
 \begin{cases}
  \varphi_{n,x}(L,\cdot) \rightarrow \varphi_x(L,\cdot) \mbox{ in } L^2(0,T), \\
 \psi_{n,x}(L,\cdot) \rightarrow \psi_x(L,\cdot) \mbox{ in } L^2(0,T)
 \end{cases}
 \end{equation}
and
 \begin{equation*}
 \begin{cases}
 \varphi_{n,xx}(L,\cdot)+\frac{ab}{c}\psi_{n,xx}(L,\cdot)  \rightarrow \varphi_{xx}(L,\cdot)+\frac{ab}{c}\psi_{xx}(L,\cdot)&  \mbox{ in } L^2(0,T), \\
  \varphi_{n,xx}(0,\cdot)+\frac{ab}{c}\psi_{n,xx}(0,\cdot)  \rightarrow \varphi_{xx}(0,\cdot)+\frac{ab}{c}\psi_{xx}(0,\cdot)&  \mbox{ in } L^2(0,T), \\
 a\varphi_{n,xx}(L,\cdot)+\frac{1}{c}\psi_{n,xx}(L,\cdot)+\frac{r}{c}\psi_{n}(L,\cdot)  \rightarrow a\varphi_{xx}(L,\cdot)+\frac{1}{c}\psi_{xx}(L,\cdot) +\frac{r}{c}\psi(L,\cdot) &\mbox{ in } L^2(0,T),\\
  a\varphi_{n,xx}(0,\cdot)+\frac{1}{c}\psi_{n,xx}(0,\cdot)+\frac{r}{c}\psi_{n}(0,\cdot)  \rightarrow a\varphi_{xx}(0,\cdot)+\frac{1}{c}\psi_{xx}(0,\cdot) +\frac{r}{c}\psi(0,\cdot) & \mbox{ in } L^2(0,T),
 \end{cases}
 \end{equation*}
as $n\rightarrow\infty$. Since $(\varphi_n,\psi_n)$ is a solution of the adjoint system, we obtain that
 \begin{equation*}
 \begin{cases}
\varphi_{xx}(L,\cdot)+\frac{ab}{c}\psi_{xx}(L,\cdot)  =0,\\
\varphi_{xx}(0,\cdot)+\frac{ab}{c}\psi_{xx}(0,\cdot)  =0,\\
a\varphi_{xx}(L,\cdot)+\frac{1}{c}\psi_{xx}(L,\cdot) +\frac{r}{c}\psi(L,\cdot)=0, \\
a\varphi_{xx}(0,\cdot)+\frac{1}{c}\psi_{xx}(0,\cdot) +\frac{r}{c}\psi(L,\cdot)=0.
 \end{cases}
 \end{equation*}
On the other hand, from \eqref{e4'} and \eqref{eeee1}, we have
\begin{equation*}
\varphi_x(L,\cdot)=\psi_x(L,\cdot)=0.
\end{equation*}
Finally, we obtain that $(\varphi,\psi)$ is a solution of
\begin{equation}\label{e11}
\begin{cases}
\varphi_t + \varphi_{xxx} + a \frac{ab}{c}\psi_{xxx}=0, & \text{in} \,\, (0,L)\times(0,T), \\
\psi_t  +\frac{r}{c}\psi_x+a\varphi_{xxx} +\frac{1}{c}\psi_{xxx} =0,  & \text{in} \,\, (0,L)\times(0,T), \\
a\varphi_{xx}(L,t)+\frac{1}{c}\psi_{xx}(L,t)+\frac{r}{c}\psi(L,t) =0, &  \text{in} \,\, (0,T), \\
a\varphi_{xx}(0,t)+\frac{1}{c}\psi_{xx}(0,t)+\frac{r}{c}\psi(0,t) =0, &  \text{in} \,\, (0,T), \\
\varphi_{xx} (L,t) +\frac{ab}{c}\psi_{xx} (L,t) =0,  &  \text{in} \,\, (0,T). \\
\varphi_{xx} (0,t) +\frac{ab}{c}\psi_{xx} (0,t) =0,  &  \text{in} \,\, (0,T). \\
\varphi_x(0,t)=\psi_x(0,t) =0, & \text{in} \,\, (0,T),\\
\varphi(x,T)= \varphi^1(x), \qquad \psi(x,T)= \psi^1(x),  &  \text{in} \,\, (0,L),
\end{cases}
\end{equation}
satisfying the additional boundary conditions
\begin{equation}\label{e12}
\varphi(0,t)=\psi(0,t)=\varphi_x(L,t)=\psi_x(L,t)=0  \,\, \text{ in } \,\,  (0,T)
\end{equation}
and
\begin{equation}\label{e13}
\|(\varphi^1,\psi^1)\|_{\X}=1.
\end{equation}
Observe that \eqref{e13} implies that the solutions of \eqref{e11}-\eqref{e12} can not be identically zero. However,  by Lemma \ref{lem}, one can conclude that $(\varphi,\psi)=(0,0)$, which drive us to a contradiction.\end{proof}

\begin{lemma}\label{lem}
For any $T > 0$, let $N_T$ denote the space of the initial states $(\varphi^1,\psi^1) \in \X$, such that the solution of \eqref{e11} satisfies \eqref{e12}. Then, $N_T=\{0\}$.
\end{lemma}
\begin{proof}
The proof uses the same arguments as those given in \cite{rosier}.

If $N_T\neq\{0\}$, the map  $(\varphi^1,\psi^1) \in N_T \rightarrow A(N_T)\subset \C N_T$
(where $\C N_T$ denote the complexification of $N_T$) has (at least) one eigenvalue. Hence, there exist $\lambda \in \C$ and $\varphi_0,\psi_0 \in  H^3(0,L)\setminus \{ 0 \}$, such that
\begin{equation*}
\begin{cases}\lambda\varphi_0+ \varphi'''_0 +\frac{ab}{c}\psi'''_0=0, & \text{in} \,\, (0,L),  \\
\lambda\psi_0 +\frac{r}{c}\psi'_0+a\varphi'''_{0} +\frac{1}{c}\psi'''_{0}=0, & \text{in} \,\, (0,L),  \\
\varphi'_0(x)=\psi'_0(x)=0, & \text{in} \,\, \{0,L\},\\
a\varphi''_{0}(x)+\frac{1}{c}\psi''_{0}(x)+\frac{r}{c}\psi_0(x) =0,  & \text{in} \,\, \{0,L\},\\
\varphi''_{0} (x) +\frac{ab}{c}\psi''_{0} (x) =0, & \text{in} \,\, \{0,L\}, \\
\varphi_0(0)=\psi_0(0)=0. \\\end{cases}
\end{equation*}
The notation $\{0,L\}$, used above, mean that the expression is applied in $0$ and $L$.

Since $1-a^2b >0$, the above system becomes
\begin{equation}\label{e14}
\begin{cases}
\lambda\varphi_0+ \varphi'''_0 + \frac{ab}{c}\psi'''_0=0, & \text{in} \,\, (0,L),  \\
\lambda\psi_0 +\frac{r}{c}\psi'_0+a\varphi'''_{0} +\frac{1}{c}\psi'''_{0}=0, & \text{in} \,\, (0,L),  \\
\varphi_0(0)=\varphi'_0(0)=\varphi''_0(0)=0, \\
\psi_0(0)=\psi'_0(0) =\psi''_0(0)=0.
\end{cases}
\end{equation}
By straightforward computations we see that $(\varphi_0,\psi_0)=(0,0)$ is a solution of \eqref{e14} for all $L > 0$, which concludes the proof of Lemma \ref{lem} and Proposition \ref{prop4}.
\end{proof}

The following theorem gives a positive answer for the control problem:

\begin{thm}\label{teo1}
Let $T > 0$ and $L>0$. Then, the system \eqref{ggln4}-\eqref{gglnb1} is exactly controllable in time T.
\end{thm}
\begin{proof}
Let us denote by $\Gamma$ the linear and bounded map  defined by
\begin{equation*}
\begin{tabular}{r c c c}
$\Gamma :$ & $L^2(0, L) \times L^2(0, L)$        &  $\longrightarrow$ & $L^2(0, L) \times L^2(0, L)$ \\
           & $(\varphi^1(\cdot), \psi^1(\cdot))$ &  $\longmapsto$     & $\Gamma(\varphi^1(\cdot), \psi^1(\cdot))=(u(\cdot,T), v(\cdot,T))$,
\end{tabular}
\end{equation*}
where $(u,v)$ is the solution of \eqref{ggln4}-\eqref{gglnb1},  with
\begin{equation}\label{contr1}
\begin{cases}
h_0(t) = (-\Delta_t)^{\frac13}\left(\varphi(0,t)+\frac{ab}{c}\psi(0,t)\right), & h_1(t)= \varphi_x(L,t)+\frac{ab}{c}\psi_x(L,t), \\
g_0(t)=(-\Delta_t)^{\frac13}\left(a\varphi(0,t)+\frac{1}{c}\psi(0,t)\right),   & g_1(t)=a\varphi_x(L,t)+\frac{1}{c}\psi_x(L,t),
\end{cases}
\end{equation}
and $(\varphi,\psi)$ the solution of the system \eqref{linadj}-\eqref{linadjbound} with $\Delta_t= \partial^2_t$ and initial data $(\varphi^1,\psi^1)$.  According to Lemma \ref{equicontrol} and Proposition \ref{prop4}, we obtain
\begin{align*}
\left (\Gamma(\varphi^1, \psi^1), (\varphi^1, \psi^1) \right)_{(L^2(0,L))^2} = &\left\|\varphi_x(L,\cdot)+\frac{ab}{c}\psi_x(L,\cdot)\right\|_{L^2(0,T)}^2 + \left\|a\varphi_x(L,\cdot)+\frac{1}{c}\psi_x(L,\cdot)\right\|_{L^2(0,T)}^2 \\
&+\left((-\Delta_t)^{\frac13}\left(\varphi(0,\cdot)+\frac{ab}{c}\psi(0,\cdot)\right), \varphi(0,\cdot)+\frac{ab}{c}\psi(0,\cdot) \right)_{L^2(0,T)}  \\
&+ \left( (-\Delta_t)^{\frac13}\left(a\varphi(0,\cdot)+\frac{1}{c}\psi(0,\cdot)\right), a\varphi(0,\cdot)+\frac{1}{c}\psi(0,\cdot) \right)_{L^2(0,T)}  \\
=&\left\|\varphi_x(L,\cdot)+\frac{ab}{c}\psi_x(L,\cdot)\right\|_{L^2(0,T)}^2 + \left\|a\varphi_x(L,\cdot)+\frac{1}{c}\psi_x(L,\cdot)\right\|_{L^2(0,T)}^2 \\
&+\left\|(-\Delta_t)^{\frac16}\left(a\varphi(0,\cdot)+\frac{1}{c}\psi(0,\cdot)\right)\right\|_{L^2(0,T)}^2\\
&+ \left\|(-\Delta_t)^{\frac16}\left(\varphi(0,\cdot)+\frac{ab}{c}\psi(0,\cdot)\right)\right\|_{L^2(0,T)}^2 \\
\geq& C^{-1} \|(\varphi^1, \psi^1)\|_{\X}^2.
\end{align*}
Thus, by the Lax-Milgram theorem, $\Gamma$ is invertible. Consequently, for given $(u^1,v^1) \in (L^2(0, L))^2$, we can define $(\varphi^1,\psi^1) := \Gamma^{-1}( u^1,v^1)$ to solve the system \eqref{linadj}-\eqref{linadjbound} and get $(\varphi,\psi)  \in \ZZ_T$. Then, if $h_0(t)$,  $h_1(t)$, $g_0(t)$ and $g_1(t)$ are given by \eqref{contr1}, the corresponding solution $(u,v)$ of the system \eqref{ggln4}-\eqref{gglnb1}, satisfies
$$(u(\cdot,0),v(\cdot,0))=(0,0)\quad \text{and} \quad (u(\cdot,T),v(\cdot,T))=(u^1(\cdot),=v^1(\cdot)).$$
\end{proof}
\begin{remark}\label{remark1}
An important question is whether the exact controllability holds, in time $T>0$, when we consider the boundary condition with another configuration, for example,
\begin{equation}\label{gglnb1'}
\left\lbrace\begin{tabular}{l l l l}
$u_{xx}(0,t) = 0$ & $u_x(L,t) = h_1(t)$ & $u_{xx}(L,t) = h_2(t)$,& in $(0,T)$,\\
$v_{xx}(0,t) =0,$ & $v_x(L,t) = g_1(t),$ & $v_{xx}(L,t) = g_2(t)$,& in $(0,T)$.
\end{tabular}\right.
\end{equation}
Observe that, in this case it would be necessary to prove that the following observability inequality
\begin{align*}
\|(\varphi^1, \psi^1)\|_{\X}^2\leq& C\left\lbrace\left  \|(-\Delta_t)^{\frac{1}{6}}\left(\varphi(L,\cdot)+\frac{ab}{c}\psi(L,\cdot)\right)\right \|_{L^2(0,T)}^2 +\left  \|\varphi_x(L,\cdot)+\frac{ab}{c}\psi_x(L,\cdot)\right \|_{L^2(0,T)}^2 \right. \notag \\
&\left. +\left \|(-\Delta_t)^{\frac{1}{6}}\left(a\varphi(L,\cdot)+\frac{1}{c}\psi(L,\cdot)\right)\right \|_{L^2(0,T)}^2 +\left  \|a\varphi_x(L,\cdot)+\frac{1}{c}\psi_x(L,\cdot)\right \|_{L^2(0,T)}^2\right\rbrace,
\end{align*}
holds for any $(\varphi^1,\psi^1)$ in $\X$, where $(\varphi,\psi)$ is  solution of \eqref{linadj}-\eqref{linadjbound} with initial data $(\varphi^1,\psi^1)$.
It can be done using Proposition \ref{hiddenregularities}  together with the contradiction argument used in the proof of Proposition \ref{prop4}. Thus, the next result about the exact controllability of the system \eqref{ggln4}-\eqref{gglnb1'} also holds:
\begin{thm}\label{teo2}
Let $T > 0$ and $L>0$. Then, the system \eqref{ggln4}-\eqref{gglnb1'} is exactly controllable in time T.
\end{thm}
\end{remark}

\subsubsection{Case 2} We consider the following boundary conditions:
\begin{equation}\label{gglnb2}
\left\lbrace\begin{tabular}{l l l l}
$u_{xx}(0,t) = 0$ & $u_x(L,t) = h_1(t)$ & $u_{xx}(L,t) = 0$,& in $(0,T)$,\\
$v_{xx}(0,t) =g_0(t),$ & $v_x(L,t) = g_1(t),$ & $v_{xx}(L,t) = g_2(t)$,& in $(0,T)$.
\end{tabular}\right.
\end{equation}
First, as in subsection above, we give an equivalent condition for the exact controllability property. It can be done using the same idea of the proof of Lemma \ref{equicontrol}.
\begin{lemma}\label{equicontrol2}
For any $(u^1,v^1)$ in $\X$, there exist four controls $\overrightarrow{h}=(0,h_1,0)$ and $\overrightarrow{g}=(g_0,g_1,g_2)$ in $\HH_T$, such that the solution $(u,v)$ of \eqref{ggln4}-\eqref{gglnb2} satisfies \eqref{finaldata} if and only if
\begin{align}
\int_0^L(u^1(x)\varphi^1(x)+v^1(x)\psi^1(x))dx=&\int_0^T g_0(t)\left(a\varphi(0,t)+\frac{1}{c}\psi(0,t)\right)dt\nonumber\\
&+\int_0^t g_1(t)\left(a\varphi_x(L,t)+\frac{1}{c}\psi_x(L,t)\right)dt\nonumber \\
&-\int_0^Tg_2(t)\left(a\varphi(L,t)+\frac{1}{c}\psi(L,t)\right)dt\label{ad3} \\
&+\int_0^T h_1(t)\left(\varphi_x(L,t)+\frac{ab}{c}\psi_x(L,t)\right)dt\nonumber,
\end{align}
for any $(\varphi^1,\psi^1)$ in $\X$, where $(\varphi,\psi)$ is the solution of the backward system \eqref{linadj}-\eqref{linadjbound}.
\end{lemma}
To prove the exact controllability property, it suffices to prove the following observability inequality:
\begin{prop}\label{prop5}
Let $T > 0$ and $L\in (0,\infty)\setminus  \FF_r$, where $\FF_r$ is given by \eqref{critical_f}. Then, there exists a constant $C (T,L) > 0$, such that
\begin{align*}
\|(\varphi^1, \psi^1)\|_{\X}^2\leq& C\left\lbrace \left  \|(-\Delta_t)^{\frac{1}{6}}\left( a\varphi(0,\cdot)+\frac{1}{c}\psi(0,\cdot)\right)\right \|_{L^2(0,T)}^2 +\left  \|\varphi_x(L,\cdot)+\frac{ab}{c}\psi_x(L,\cdot)\right \|_{L^2(0,T)}^2 \right. \notag \\
&\left. +\left \|(-\Delta_t)^{\frac{1}{6}}\left(a\varphi(L,\cdot)+\frac{1}{c}\psi(L,\cdot)\right)\right \|_{L^2(0,T)}^2 +\left  \|a\varphi_x(L,\cdot)+\frac{1}{c}\psi_x(L,\cdot)\right \|_{L^2(0,T)}^2\right\rbrace,
\end{align*}
for any $(\varphi^1,\psi^1)$ in $\X$, where $(\varphi,\psi)$ is  solution of \eqref{linadj}-\eqref{linadjbound} with initial data $(\varphi^1,\psi^1)$, where $\Delta_t :=\partial_t^2$.
\end{prop}

\begin{proof}
 We proceed as in the proof of Proposition \ref{prop4} using the contradiction argument. Therefore, we will summarize it. Firstly, we show that the sequences $\{(\varphi^1_n, \psi^1_n)\}_{n \in \N}$, $$\{a\varphi_n(0,\cdot) +\frac{1}{c}\psi_n(0,\cdot)\}_{n\in \N},$$ $$\{a\varphi_n(L,\cdot) +\frac{1}{c}\psi_n(L,\cdot)\}_{n\in \N},$$ $$\{a\varphi_{n,x}(L,\cdot) +\frac{1}{c}\psi_{n,x}(L,\cdot)\}_{n\in \N}$$ and $$\{\varphi_{n,x}(L,\cdot) +\frac{ab}{c}\psi_{n,x}(L,\cdot)\}_{n\in \N},$$  are relatively compact in $\X$ and $L^2(0,T;\X)$, respectively. Next, we proceed as in the proof of Proposition \ref{prop4} to get that
$$a\varphi_n(0,\cdot) +\frac{1}{c}\psi_n(0,\cdot)\rightarrow0,$$ $$a\varphi_n(L,\cdot) +\frac{1}{c}\psi_n(L,\cdot)\rightarrow0,$$ $$\varphi_{n,x}(L,\cdot)\rightarrow0,\,\,\psi_x(L,\cdot)\rightarrow0,$$
as $n\rightarrow\infty$, and $$||(\varphi,\psi)||_{(L^2(0,L))^2}=1.$$
Finally, combining the hidden regularity of the solutions of the adjoint system \eqref{hr4} and the compact embedding $H^{\frac13}(0,T)\hookrightarrow L^2(0,T)$, we conclude that $(\varphi,\psi)$ satisfies
\begin{equation}\label{e17}
\begin{cases}
\varphi_t + \varphi_{xxx} +\frac{ab}{c}\psi_{xxx}=0 & \mbox{ in }(0,L)\times(0,T),\\
\psi_t  +\frac{r}{c}\psi_x+a\varphi_{xxx} +\frac{1}{c}\psi_{xxx} =0 & \mbox{ in } (0,L)\times(0,T), \\
\varphi_{xx}(L,t) +\frac{ab}{c}\psi_{xx}(L,t)=0 & \mbox{ in }(0,T),\\
\varphi_{xx}(0,t) +\frac{ab}{c}\psi_{xx}(0,t)=0 & \mbox{ in }(0,T),\\
a\varphi_{xx}(L,t) +\frac{1}{c}\psi_{xx}(L,t) +\frac{r}{c}\psi(L,t)=0 & \mbox{ in }(0,T),\\
a\varphi_{xx}(0,t) +\frac{1}{c}\psi_{xx}(0,t) +\frac{r}{c}\psi(0,t)=0 & \mbox{ in }(0,T),\\
\varphi_x(0,t)=\psi_x(0,t)=0 & \mbox{ in }(0,T),\\
\varphi(x,T)= \varphi^1(x), \,\,
\psi(x,T)= \psi^1(x) &  \mbox{ in } (0,L)
\end{cases}
\end{equation}
and
\begin{equation}\label{e17'}
\begin{cases}
a\varphi(L,t) +\frac{1}{c}\psi(L,t)=0 & \mbox{ in }(0,T), \\
a\varphi(0,t) +\frac{1}{c}\psi(0,t)=0 & \mbox{ in }(0,T), \\
\varphi_x(L,t)=\psi_x(L,t)=0 & \mbox{ in }(0,T), \\
\|(\varphi,\psi)\|_{\X} = 1.
\end{cases}
\end{equation}

Notice that  the solutions of \eqref{e17}-\eqref{e17'} can not be identically zero. Therefore, from Lemma \ref{lem1}, one can conclude that $(\varphi,\psi)=(0,0)$, which drive us to a contradiction.
\end{proof}
\begin{lemma}\label{lem1}
For any $T > 0$, let $N_T$ denote the space of the initial states $(\varphi^1,\psi^1) \in \X$, such that the solution of \eqref{e17} satisfies \eqref{e17'}. Then, for $L \in (0,\infty)\setminus \FF_r$,  $N_T=\{0\}$.
\end{lemma}
\begin{proof}
By the same arguments given in \cite{rosier}, if $N_T\neq \{0\}$, the map  $(\varphi^1,\psi^1) \in N_T \rightarrow A(N_T)\subset \C N_T$
has (at least) one eigenvalue. Hence,  there exist $\lambda \in \C$ and $\varphi_0,\psi_0 \in  H^3(0,L)\setminus \{ 0 \}$, such that
\begin{equation}\label{e18}
\begin{cases}
\lambda\varphi_0+ \varphi'''_0 + \frac{ab}{c}\psi'''_0=0, & \text{in} \,\, (0,L),  \\
\lambda\psi_0 +\frac{r}{c}\psi'_0+a\varphi'''_{0} +\frac{1}{c}\psi'''_{0}=0, & \text{in} \,\, (0,L),  \\
a\varphi_0(x)+\frac{1}{c}\psi_0(x)=0, & \text{in} \,\, \{0,L\},\\
\varphi'_0(x)=\psi'_0(x)=0, & \text{in} \,\, \{0,L\},\\
\varphi''_{0} (x) +\frac{ab}{c}\psi''_{0} (x) =0, & \text{in} \,\, \{0,L\}, \\
a\varphi''_{0}(x)+\frac{1}{c}\psi''_{0}(x)+\frac{r}{c}\psi_0(x) =0, & \text{in} \,\, \{0,L\}.
\end{cases}
\end{equation}
To conclude the proof of the Lemma \ref{lem1}, we prove that this does not hold if $L \in (0,\infty) \setminus \FF_r$. To simplify the notation, henceforth we denote $(\varphi_0,\psi_0):=(\varphi,\psi)$.

\begin{lemma}\label{lem2}
Let $L>0$. Consider the assertion
\begin{equation*}
(\NN):\ \ \exists \lambda \in \C,  \exists \varphi,\psi \in H^3(0,L)\setminus(0,0),\,\, \text{such that}\,\,
\begin{cases}
\lambda\varphi+ \varphi''' + \frac{ab}{c}\psi'''=0, & \text{in} \,\, (0,L),  \\
\lambda\psi +\frac{r}{c}\psi'+a\varphi''' +\frac{1}{c}\psi'''=0, & \text{in} \,\, (0,L),  \\
a\varphi(x)+\frac{1}{c}\psi(x)=0, & \text{in} \,\, \{0,L\},\\
\varphi'(x)=\psi'(x)=0, & \text{in} \,\, \{0,L\},\\
\varphi'' (x) +\frac{ab}{c}\psi''(x) =0, & \text{in} \,\, \{0,L\}, \\
a\varphi''(x)+\frac{1}{c}\psi''(x)+\frac{r}{c}\psi(x) =0,  & \text{in} \,\, \{0,L\}.
\end{cases}
\end{equation*}
Then, $(\NN)$ holds if and only if $L \in \FF_r$.
\end{lemma}
\noindent\textit{Proof.}
We use an argument similar to the one used in \cite[Lemma 3,5]{rosier}. Let us introduce the notation $\hat{\varphi}(\xi) =\int_0^L e^{-i x \xi}\varphi(x)dx$ and $\hat{\psi}(\xi) =\int_0^L e^{-i x \xi}\psi(x)dx$. Then, multiplying the first and the second equations in $(\NN)$ by $e^{-i x \xi}$ and integrating by part in $(0,L)$, 
it follows that
\begin{multline*}
\left( (i\xi)^3+\lambda\right) \hat{\varphi}(\xi) +\frac{ab}{c}(i \xi)^3\hat{\psi}(\xi) \\
+ \left[ \left(\left(\varphi''(x)+\frac{ab}{c}\psi''(x)\right) +(i\xi)\left(\varphi'(x)+\frac{ab}{c}\psi'(x)\right)+(i\xi)^2\left(\varphi(x)+\frac{ab}{c}\psi(x)\right)\right) e^{-ix\xi}\right]_0^L = 0
\end{multline*}
and
\begin{multline*}
\left( \frac{1}{c}(i\xi)^3+\frac{r}{c}(i\xi)+\lambda\right) \hat{\psi}(\xi) +a(i \xi)^3\hat{\varphi}(\xi) + \left[  \left(\left(a\varphi''(x)+\frac{1}{c}\psi''(x)+\frac{r}{c}\psi(\xi)\right)\right.\right. \\
\left.\left.+(i\xi)\left(a\varphi'(x)+\frac{1}{c}\psi'(x)\right)+(i\xi)^2\left(a\varphi(x)+\frac{1}{c}\psi(x)\right)\right) e^{-ix\xi}\right]_0^L = 0.
\end{multline*}
The boundary conditions allow us to conclude that
\begin{equation}\label{eqq1}
\begin{cases}
[(i\xi)^3+\lambda]\hat{\varphi}(\xi)+\dfrac{ab}{c}(i\xi)^3\hat{\psi}(\xi) = (i\xi)^2\left(\varphi(0)+\dfrac{ab}{c}\psi(0)- \left(\varphi(L)+\dfrac{ab}{c}\psi(L)\right)e^{-iL\xi}\right), \\
\\
\dfrac{1}{c}[(i\xi)^3+r(i\xi)+c\lambda]\hat{\psi}(\xi) + a(i\xi)^3\hat{\varphi}(\xi)=0.
\end{cases}
\end{equation}
Then, from the first equation in \eqref{eqq1}, we obtain
\begin{equation}\label{eqq2}
\hat{\varphi}(\xi) = \frac{(i\xi)^2\left(\alpha+\beta e^{-iL\xi}\right)}{(i\xi)^3+\lambda} -\frac{ab(i\xi)^3\hat{\psi}(\xi)}{c\left((i\xi)^3+\lambda\right)},
\end{equation}
where $\alpha= \varphi(0)+\frac{ab}{c}\psi(0)$ and $\beta=-\varphi(L)-\frac{ab}{c}\psi(L)$. Replacing the above expression in the second equation in \eqref{eqq1} it follows that $$ \frac{1}{c}\left[ (i\xi)^3+r(i\xi)+c\lambda - \frac{a^2b(i\xi)^6}{(i\xi)^3+\lambda}\right]\hat{\psi}(\xi) =-\frac{a(i\xi)^5\left(\alpha+\beta e^{-iL\xi}\right)}{(i\xi)^3+\lambda}.$$
Thus,
\begin{equation}\label{eqq3}
\hat{\psi}(\xi) =-\frac{ac(i\xi)^5\left(\alpha+\beta e^{-iL\xi}\right)}{(1-a^2b)(i\xi)^6+r(i\xi)^4+(c+1)\lambda(i\xi)^3+r\lambda(i\xi)+c\lambda^2}.
\end{equation}
Replacing \eqref{eqq3} in \eqref{eqq2}, we obtain
\begin{align*}
\hat{\varphi}(\xi)& =  \frac{ (i\xi)^2\left( (i\xi)^3+r(i\xi)+c\lambda \right)\left(\alpha+\beta e^{-iL\xi}\right)}{(1-a^2b)(i\xi)^6+r(i\xi)^4+(c+1)\lambda(i\xi)^3+r\lambda(i\xi)+c\lambda^2}.
\end{align*}
Setting $\lambda=ip$, $p \in \C$, from the previous identities we can write $\hat{\psi}(\xi)=-i\left[ acf(\xi)\right]$ and $\hat{\varphi}(\xi)=-ig(\xi)$, where
\begin{equation*}
\begin{cases}
f(\xi) = \dfrac{\xi^5\left( \alpha +\beta e^{-iL\xi}\right)}{P(\xi)}, \\
\\
g(\xi)= \dfrac{\xi^2\left(\xi^3-r\xi-cp\right)\left( \alpha +\beta e^{-iL\xi}\right)}{P(\xi)},
\end{cases}
\end{equation*}
with
\begin{equation*}
P(\xi):=(1-a^2b)\xi^6-r\xi^4-(c+1)p\xi^3+rp\xi+cp^2.
\end{equation*}
Using Paley-Wiener theorem (see \cite[Section 4, page 161]{yosida}) and the usual characterization of $H^2(\R)$ functions by means of their Fourier transforms, we see that $(\NN)$ is equivalent to the existence of $p \in \C$ and $(\alpha,\beta) \in \C^2 \setminus (0,0),$ such that
\begin{enumerate}
\item[(i)] $f$ and $g$ are entire functions in $\C$,
\item[(ii)] $\displaystyle\int_{\R}|f(\xi)|^2(1+|\xi|^2)^2d\xi<\infty$ and $\int_{\R}|g(\xi)|^2(1+|\xi|^2)^2d\xi<\infty$,
\item[(iii)] $\forall \xi \in \C$, we have that $|f(\xi)|\leq c_1(1+|\xi|)^ke^{L|Im\xi|}$ and  $|g(\xi)|\leq c_1(1+|\xi|)^ke^{L|Im\xi|}$, for some positive constants $c_1 $ and $k$.
\end{enumerate}
Notice that if (i) holds true, then (ii) and (iii) are satisfied. Recall that $f$ and $g$ are entire functions if and only if the roots $\xi_0, \xi_1, \xi_2, \xi_3, \xi_4$ and $\xi_5$ of $P(\xi)$ are roots of $\xi^5\left( \alpha +\beta e^{-iL\xi}\right)$ and $\xi^2(\xi^3-r\xi-cp)\left( \alpha +\beta e^{-iL\xi}\right)$.
\vglue 0.3 cm
Let us first assume that $\xi=0$ is not a root of $P(\xi)$.  Thus, it is sufficient to consider the case when $\alpha +\beta e^{-iL\xi}$ and $P(\xi)$ share the same roots. Observe that the roots of $\alpha +\beta e^{-iL\xi}$ are simple, unless $\alpha = \beta= 0$ (Indeed, in this case $\varphi(0)+\frac{ab}{c}\psi(0)=0$ and $\varphi(L)+\frac{ab}{c}\psi(L)=0$ and using the system \eqref{e18} we conclude that $(\varphi,\psi)=(0,0)$, which is a contradiction). Then, (i) holds provided that the roots of $P(\xi)$ are simple. Therefore, it follows that $(\NN)$ is equivalent to the existence of complex numbers $p$ and $\xi_0$ and positive integers $k,l,m,n$ and $s$, such that, if we set
\begin{equation}\label{eqq4}
\xi_1=\xi_0+\frac{2\pi}{L}k, \quad \xi_2=\xi_1+\frac{2\pi}{L}l, \quad \xi_3=\xi_2+\frac{2\pi}{L}m, \quad \xi_4=\xi_3+\frac{2\pi}{L}n\quad \text{and} \quad \xi_5=\xi_4+\frac{2\pi}{L}s,
\end{equation}
$P(\xi)$ can be written as follows
\begin{equation*}
P(\xi)=(\xi-\xi_0)(\xi-\xi_1)(\xi-\xi_2)(\xi-\xi_3)(\xi-\xi_4)(\xi-\xi_5).
\end{equation*}
In particular, we obtain the following relations:
\begin{gather}
\xi_0+\xi_1+\xi_2+\xi_3+\xi_4+\xi_5=0, \label{eqq6}
\end{gather}
\begin{multline}
\xi_0(\xi_1+\xi_2+\xi_3+\xi_4+\xi_5)+\xi_1(\xi_2+\xi_3+\xi_4+\xi_5)+\xi_2(\xi_3+\xi_4+\xi_5) \\
+\xi_3(\xi_4+\xi_5) +\xi_4\xi_5=-\frac{r}{1-a^2b}\label{eqq7}
\end{multline}
and
\begin{gather*}
\xi_0\xi_1\xi_2\xi_3\xi_4\xi_5=\left(\frac{c}{1-a^2b}\right) p^2. \label{eqq8}
\end{gather*}
\eqref{eqq4} and \eqref{eqq6} imply that
\begin{multline*}
\xi_0+\left(\xi_0+\frac{2\pi}{L}k\right)+\left(\xi_0+\frac{2\pi}{L}(k+l)\right)+\left(\xi_0+\frac{2\pi}{L}(k+l+m)\right)+\left(\xi_0+\frac{2\pi}{L}(k+l+m+n)\right) \\
+\left(\xi_0+\frac{2\pi}{L}(k+l+m+n+s)\right)=0.
\end{multline*}
Straightforward computations lead to
\begin{equation}\label{eqq9}
 \xi_0 = -\dfrac{\pi}{3L}(5k+4l+3m+2n+s).
\end{equation}
On the other hand, from \eqref{eqq7}, we obtain
\begin{multline*}
\xi_0\left(5\xi_0+\frac{2\pi}{L}(5k+4l+3m+2n+s)\right) +\left(\xi_0+\frac{2\pi}{L}k\right) \left(4\xi_0+\frac{2\pi}{L}(4k+4l+3m+2n+s)\right) \\
 +\left(\xi_0+\frac{2\pi}{L}(k+l)\right) \left(3\xi_0+\frac{2\pi}{L}(3k+3l+3m+2n+s)\right)  \\
  +\left(\xi_0 +\frac{2\pi}{L}(k+l+m)\right) \left(2\xi_0+\frac{2\pi}{L}(2k+2l+2m+2n+s)\right) \\
  +\left(\xi_0+\frac{2\pi}{L}(k+l+m+n)\right) \left(\xi_0+\frac{2\pi}{L}(k+l+m+n+s)\right) = -\frac{r}{1-a^2b}.
\end{multline*}
Thus, we have
\begin{equation}\label{eqq12}
15\xi_0^2+\frac{2\pi}{L}(25k+20+15m+10n+5s)\xi_0+\frac{4\pi^2}{L^2}\eta = -\frac{r}{1-a^2b},
\end{equation}
where
\begin{multline*}
 \eta=k(10k+10l+9m+7n+4s) +l(6k+6l+6m+5n+3s)  \\
 +m(3k+3l+3m+3n+2s)+n(k+l+m+n+s).
\end{multline*}
Replacing \eqref{eqq9} in \eqref{eqq12}, we obtain
\begin{equation*}
\frac{3rL^2}{1-a^2b}= \pi^2\left(5(5k+4l+3m+2n+s)^2-12\eta\right).
\end{equation*}
From the discussion above, we can conclude that
\begin{equation}\label{eqq14}
\begin{cases}
L=\pi \sqrt{\dfrac{(1-a^2b)\alpha(k,l,m,n,s)}{3r}},\\
\\
\xi_0 = -\dfrac{\pi}{3}(5k+4l+3m+2n+s),\\
\\
p= \sqrt{\dfrac{(1-a^2b)\xi_0\xi_1\xi_2\xi_3\xi_4\xi_5}{c}},
\end{cases}
\end{equation}
where
\begin{align*}
\alpha(k,l,m,n,s):=&5k^2+8l^2+9m^2+8n^2+5s^2+8kl+6km+4kn+2ks+12ml \\
&+8ln+3ls+12mn+6ms+8ns.
\end{align*}
Now, we assume that $\xi_0= 0$ is a root of $P(\xi)$. Then, it follows that $p=0$ and
\begin{equation*}
\begin{cases}
f(\xi)= \dfrac{\xi^5\left(\alpha+\beta e^{-iL\xi}\right)}{(1-a^2b)\xi^6-r\xi^4}=\dfrac{\xi\left(\alpha+\beta e^{-iL\xi}\right)}{(1-a^2b)\xi^2-r},\\
\\
g(\xi)= \dfrac{\xi^2\left(\xi^3-r\xi\right)\left( \alpha +\beta e^{-iL\xi}\right)}{(1-a^2b)\xi^6-r\xi^4}=\dfrac{\left(\xi^2-r\right)\left( \alpha +\beta e^{-iL\xi}\right)}{\xi\left((1-a^2b)\xi^2-r\right)}.
\end{cases}
\end{equation*}
In this case, $(\NN)$ holds if and only if $f$ and $g$ satisfy (i), (ii) and (iii). Thus, (i) holds provided that
\begin{equation*}
\xi_0=0, \quad \xi_1=\sqrt{\frac{r}{1-a^2b}} \quad \text{and} \quad \xi_2=-\sqrt{\frac{r}{1-a^2b}}
\end{equation*}
are roots of $\alpha+\beta e^{-iL\xi}$. Therefore, we can write $\xi_1=\xi_0+\frac{2\pi}{L}k$, for $k \in \Z$. Consequently, it follows that
\begin{equation}\label{eqq15}
L= 2\pi k \sqrt{\dfrac{1-a^2b}{r}}.
\end{equation}
Finally, from \eqref{eqq14} and \eqref{eqq15}, we deduce that $(\NN)$ holds if and only if $L \in \FF_r$, where $\FF_r$ is given by \eqref{critical_f}. This completes the proof of Lemma \ref{lem2}, Lemma \ref{lem1} and, consequently, the proof of Proposition \ref{prop5}.
\end{proof}
The next result gives a positive answer for the control problem, and can be proved using the same ideas presented in Theorem \ref{teo1} and, thus, we will omit the proof.
\begin{thm}\label{teo3}
Let $T > 0$ and $L \in (0,\infty) \setminus \FF_r$, where $\FF_r$ is given by \eqref{critical_f}. Then, the system \eqref{ggln4}-\eqref{gglnb2} is exactly controllable in time T.
\end{thm}

\begin{remark}\label{remark2}
As in the previous subsection, the question here is weather system \eqref{ggln4}-\eqref{gglnb2'} is exactly controllable with another configuration of the boundary condition, for example,
\begin{equation}\label{gglnb2'}
\left\lbrace\begin{tabular}{l l l l}
$u_{xx}(0,t) = h_0(t)$, & $u_x(L,t) = h_1(t)$, & $u_{xx}(L,t) = h_2(t)$& in $(0,T)$,\\
$v_{xx}(0,t) =0,$ & $v_x(L,t) = g_1(t),$ & $v_{xx}(L,t) = 0$& in $(0,T)$.
\end{tabular}\right.
\end{equation}
The answer for this question is positive if we prove that the following observability inequality
\begin{align*}
\|(\varphi^1, \psi^1)\|_{\X}^2\leq C&\left\lbrace\left  \|(-\Delta)^{\frac{1}{6}}\left(\varphi(0,\cdot)+\frac{ab}{c}\psi(0,\cdot)\right)\right \|_{L^2(0,T)}^2 +\left  \|\varphi_x(L,\cdot)+\frac{ab}{c}\psi_x(L,\cdot)\right \|_{L^2(0,T)}^2 \right. \notag \\
&\left. +\left \|(-\Delta)^{\frac{1}{6}}\left(\varphi(L,\cdot)+\frac{ab}{c}\psi(L,\cdot)\right)\right \|_{L^2(0,T)} +\left  \|a\varphi_x(L,\cdot)+\frac{1}{c}\psi_x(L,\cdot)\right \|_{L^2(0,T)}^2\right\rbrace,
\end{align*}
holds, for any $(\varphi^1,\psi^1)$ in $\X$, where $(\varphi,\psi)$ is  solution of \eqref{linadj}-\eqref{linadjbound} with initial data $(\varphi^1,\psi^1)$.
Note that it can be proved using Proposition \ref{hiddenregularities} together with the contradiction argument as in the proof of Proposition \ref{prop5}. 
Thus, the exact controllability result is also true in this case.
\begin{thm}\label{teo4}
Let $T > 0$ and $L \in (0,\infty) \setminus \FF_r$. Then, the system \eqref{ggln4}-\eqref{gglnb2'} is exactly controllable in time T.
\end{thm}
\end{remark}

\subsection{One Control}

In this subsection, we intend to prove the exact controllability of the system by using only one boundary control $h_1$ or  $g_1$ and fixing $h_0=h_2=g_0=g_2=0$, namely,
\begin{equation}\label{gglnb3}
\left\lbrace\begin{tabular}{l l l l}
$u_{xx}(0,t) = 0$ & $u_x(L,t) = h_1(t),$ & $u_{xx}(L,t) = 0$,& in $(0,T)$,\\
$v_{xx}(0,t) =0,$ & $v_x(L,t) = 0,$ & $v_{xx}(L,t) =0$,& in $(0,T)$.
\end{tabular}\right.
\end{equation}
or
\begin{equation}\label{gglnb3'}
\left\lbrace\begin{tabular}{l l l l}
$u_{xx}(0,t) = 0$ & $u_x(L,t) = 0$ & $u_{xx}(L,t) = 0$,& in $(0,T)$,\\
$v_{xx}(0,t) =0,$ & $v_x(L,t) = g_1(t),$ & $v_{xx}(L,t) =0$,& in $(0,T)$.
\end{tabular}\right.
\end{equation}
The result below give us an equivalent condition for the exact controllability and the proof is analogous to the proof of the Lemma \ref{equicontrol}.
\begin{lemma}\label{equicontrol3}
For any $(u^1,v^1)$ in $\X$, there exist one control $\overrightarrow{h}=(0,h_1,0)$ and $\overrightarrow{g}=(0,0,0)$ (resp. $\overrightarrow{h}=(0,0,0)$ and $\overrightarrow{g}=(0,g_1,0)$) in $\HH_T$, such that the solution $(u,v)$ of \eqref{ggln4}-\eqref{gglnb3} (resp. \eqref{ggln4}-\eqref{gglnb3'})  satisfies \eqref{finaldata} if and only if
\begin{multline*}
\int_0^L(u^1(x)\varphi^1(x)+v^1(x)\psi^1(x))dx= \int_0^T h_1(t)\left[\varphi_x(L,t)+\frac{ab}{c}\psi_x(L,t)\right]dt \\
\left(\text{resp.}\quad \int_0^L(u^1(x)\varphi^1(x)+v^1(x)\psi^1(x))dx= \int_0^T g_1(t)\left[a\varphi_x(L,t)+\frac{1}{c}\psi_x(L,t)\right]dt \right)
\end{multline*}
for any $(\varphi^1,\psi^1)$ in $\X$, where  $(\varphi,\psi)$  is the solution of the backward system \eqref{linadj}-\eqref{linadjbound}.
\end{lemma}
Note that using the change of variable $x' = L-x$ and $t' = T-t$, the system \eqref{linadj}-\eqref{linadjbound} is equivalent to the following back-forward system
\begin{align}
&\begin{cases}\label{linadj1}
\varphi_t +\varphi_{xxx} +\frac{ab}{c}\psi_{xxx}=0,  & \text{in} \,\, (0,L)\times (0,T),\\
\psi_t  +\frac{r}{c}\psi_x+a\varphi_{xxx} +\frac{1}{c}\psi_{xxx} =0,  & \text{in} \,\, (0,L)\times (0,T), \\
\varphi(x,0)= \varphi^0(x), \,\,
\psi(x,0)= \psi^0(x), & \text{in} \,\, (0,L),
\end{cases}
\end{align}
with boundary conditions
\begin{equation}\label{linadjbound1}
\begin{cases}
\varphi_{xx} (x,t) +\frac{ab}{c}\psi_{xx} (x,t) =0, & \text{in} \,\, \{0,L\}\times (0,T),\\
a\varphi_{xx}(x,t)+\frac{1}{c}\psi_{xx}(x,t)+\frac{r}{c}\psi(x,t) =0, & \text{in} \,\, \{0,L\}\times (0,T), \\
\varphi_x(L,t)=\psi_x(L,t) =0, & \text{in} \,\, (0,T).
\end{cases}
\end{equation}
It is well know (according to the previous sections) that the observability inequality
\begin{align}\label{obineq3}
\|(\varphi^0, \psi^0)\|_{\X}^2\leq C  \left  \|\varphi_x(0,\cdot)+\frac{ab}{c}\psi_x(0,\cdot)\right \|_{L^2(0,T)}^2
\end{align}
or 
\begin{align}\label{obineq3'}
\|(\varphi^0, \psi^0)\|_{\X}^2\leq C \left \|a\varphi_x(0,\cdot)+\frac{1}{c}\psi_x(0,\cdot)\right\|_{L^2(0,T)}^2
\end{align}
plays a fundamental role for the study of the controllability. To prove \eqref{obineq3} (resp. \eqref{obineq3'}), we use a direct approach based on the multiplier technique that gives us the observability inequality for small values of the length $L$ and large time of control $T$.
\begin{prop}\label{prop6}
Let us suppose that  $T > 0$ and $L>0$ satisfy
\begin{align}\label{Lsmall}
1>\frac{\beta C_T}{T}\left[L +\frac{r}{c} \right],
\end{align}
where $C_T$ is the constant in \eqref{hr4} and $\beta$ is the constant given by the embedding $H^{\frac{1}{3}}(0,T) \subset L ^2(0,T)$. 
Then, there exists a constant $C (T,L) > 0$, such that for any $(\varphi^0,\psi^0)$ in $\X$ the observability inequality \eqref{obineq3} (resp. \eqref{obineq3'}) holds, where $(\varphi,\psi)$ is  solution of \eqref{linadj1}-\eqref{linadjbound1} with initial data $(\varphi^0,\psi^0)$.
\end{prop}

\begin{proof}
We multiply the first equation in \eqref{linadj1} by $(T-t)\varphi$, the second one by $\frac{b}{c}(T-t)\psi$ and integrate over $(0,T)\times(0,L)$, to give us:

\begin{align*}
\frac{T}{2}\int_0^L\left( \varphi_0^2(x) +\frac{b}{c}\psi_0^2(x)\right)dx=&\frac{1}{2}\int_0^T\int_0^L \left(\varphi^2(x,t)+\frac{b}{c}\psi^2(x,t)\right)dxdt \\
&+\int_0^T (T-t)\left [ \varphi(L,t)\left(\varphi_{xx}(L,t)+\frac{ab}{c}\psi_{xx}(L,t)\right)\right]dt\\
&-\int_0^T (T-t)\left [ \varphi(0,t)\left(\varphi_{xx}(0,t)+\frac{ab}{c}\psi_{xx}(0,t)\right)\right]dt \\
&+\int_0^T (T-t)\left [\frac{b}{c}\psi(L,t)\left(a\varphi_{xx}(L,t)+\frac{\psi_{xx}(L,t)}{c}+\frac{r}{2c}\psi(L,t)\right)\right]dt \\
&+\int_0^T (T-t)\left [ -\frac{b}{c}\psi(0,t)\left(a\varphi_{xx}(0,t)+\frac{\psi_{xx}(0,t)}{c}+\frac{r}{2c}\psi(0,t)\right) \right]dt \\
&+\frac12\int_0^T(T-t)\left[\varphi_x^2(0,t)+\frac{2ab}{c}\psi_x(0,t)\varphi_x(0,t)+\frac{b}{c^2}\psi_x^2(0,t)\right]dt.
\end{align*}
From the boundary conditions \eqref{linadjbound1}, we have that
\begin{align*}
\|(\varphi^0,\psi^0)\|_{\X}^2\leq& \frac{1}{T}\|(\varphi,\psi)\|_{L^2(0,T;\X)}^2 +\frac{br}{c^2T}\|\psi(0,\cdot)\|_{L^2(0,T)}^2-\frac{br}{c^2}\int_0^T\frac{T-t}{T}\psi(L,t)^2dt\\
& +\int_0^T\left[\varphi_x^2(0,t)+\frac{2ab}{c}\psi_x(0,t)\varphi_x(0,t)+\frac{b}{c^2}\psi_x^2(0,t)\right]dt,
\\
\leq&\frac{1}{T}\|(\varphi,\psi)\|_{L^2(0,T;\X)}^2 +\frac{\beta b r}{c^2T}\|\psi(0,\cdot)\|_{H^{\frac{1}{3}}(0,T)}^2 +\frac{1}{a^2b}\left\|\varphi_x(0,\cdot)+\frac{ab}{c}\psi_x(0,\cdot)\right\|_{L^2(0,T)}^2,
\end{align*}
\begin{equation*}
\left(\text{resp.} \,\, \|(\varphi^0,\psi^0)\|_{\X}^2\leq\frac{1}{T}\|(\varphi,\psi)\|_{L^2(0,T;\X)}^2 +\frac{\beta b r}{c^2T}\|\psi(0,\cdot)\|_{H^{\frac{1}{3}}(0,T)}^2 +\frac{1}{a^2}\left\|a\varphi_x(0,\cdot)+\frac{1}{c}\psi_x(0,\cdot)\right\|_{L^2(0,T)}^2 \right)
\end{equation*}
where $\beta$ is the constant given by the compact embedding $H^{\frac{1}{3}}(0,T) \subset L ^2(0,T)$. 
On the other hand, note that $L^{\infty}(0,L) \subset L^{2}(0,L)$, thus
\begin{equation}
\|\varphi(\cdot,t)\|^2_{L^2(0,L)} \leq L \|\varphi(\cdot,t)\|^2_{L^{\infty}(0,L)}, \quad \text{and} \quad \|\psi(\cdot,t)\|^2_{L^2(0,L)} \leq L \|\psi(\cdot,t)\|^2_{L^{\infty}(0,L)},  
\end{equation} 
Hence, 
\begin{align*}
\|(\varphi,\psi)\|_{L^2(0,T;\X)}^2 &= \int_0^T \left\lbrace \|\varphi(\cdot,t)\|^2_{L^2(0,L)}  + \frac{b}{c}\|\psi(\cdot,t)\|^2_{L^2(0,L)}\right\rbrace dt \\
&\leq L\int_0^T \left\lbrace \|\varphi(\cdot,t)\|^2_{L^{\infty}(0,L)}  + \frac{b}{c}\|\psi(\cdot,t)\|^2_{L^{\infty}(0,L)}\right\rbrace dt \\
&\leq  L\beta \|\varphi\|_{H^{\frac{1}{3}}(0,T;L^{\infty}(0,L))}^2+\frac{bL\beta}{c}\|\psi\|_{H^{\frac{1}{3}}(0,T;L^{\infty}(0,L))}^2 .  \\
\end{align*}
Thanks to the Proposition \ref{hiddenregularities}, we obtain
\begin{align*}
\|(\varphi^0,\psi^0)\|_{\X}^2 \leq & \frac{L \beta C_T}{T} \|\varphi^0\|_{L^2(0,L)}^2+\frac{b L \beta C_T}{cT}\|\psi^0\|_{L^2(0,L)}^2 +\frac{\beta C_T b r}{c^2T}\|\psi^0\|_{L^2(0,L)}^2 \\
& +\frac{1}{a^2b}\left\|\varphi_x(0,\cdot) +\frac{ab}{c}\psi_x(0,\cdot)\right\|_{L^2(0,T)}^2 \\
\leq &\frac{L \beta C_T}{T} \|(\varphi^0,\psi^0)\|_{\X}^2 +\frac{\beta C_T r}{cT}\|(\varphi^0,\psi^0)\|_{\X}^2  +\frac{1}{a^2b}\left\|\varphi_x(0,\cdot) +\frac{ab}{c}\psi_x(0,\cdot)\right\|_{L^2(0,T)}^2.  \\
\end{align*}
Finally, it follows that
\begin{align*}
\|(\varphi^0,\psi^0)\|_{\X}^2 \leq & K \left\|\varphi_x(0,\cdot)+\frac{ab}{c}\psi_x(0,\cdot)\right\|_{L^2(0,T)}^2
\end{align*}
under the condition
\begin{equation}\label{TLcondition}
 K= \frac{1}{a^2b}\left(1- \frac{\beta C_T}{T}\left[L +\frac{r}{c} \right]\right)^{-1}>0.
\end{equation}
\end{proof}

From the observability inequality \eqref{obineq3}, the following result holds.

\begin{thm}\label{teo5}
Let $T > 0$ and $L >0$ satisfying \eqref{Lsmall}. Then, the system \eqref{ggln4}-\eqref{gglnb3} (resp. \eqref{ggln4}-\eqref{gglnb3'}) is exactly controllable in time T.
\end{thm}
\begin{proof}
Consider the map
\begin{equation*}
\begin{tabular}{r c c c}
$\Gamma :$ & $L^2(0, L) \times L^2(0, L)$        &  $\longrightarrow$ & $L^2(0, L) \times L^2(0, L)$ \\
           & $(\varphi^1(\cdot), \psi^1(\cdot))$ &  $\longmapsto$     & $\Gamma(\varphi^1(\cdot), \psi^1(\cdot))=(u(\cdot,T), v(\cdot,T))$
\end{tabular}
\end{equation*}
where $(u,v)$ is the solution of \eqref{ggln4}-\eqref{gglnb2},  with
\begin{equation*}
\begin{cases}
h_1(t)=\varphi_x(L,t)+\frac{ab}{c}\psi_x(L,t), \\
g_1(t)= a\varphi_x(L,t)+\frac{1}{c}\psi_x(L,t),
\end{cases}
\end{equation*}
 and $(\varphi,\psi)$ is the solution of the system \eqref{linadj}-\eqref{linadjbound} with initial data $(\varphi^1,\psi^1)$. By \eqref{obineq3} (resp. \eqref{obineq3'}) and the Lax-Milgram theorem, the proof is achieved.
\end{proof}

\section{The Nonlinear Control System}\label{Sec4}
We are now in position to prove our main result considering several configurations of the control in the boundary conditions.	 Let $T>0$, from Theorems \ref{teo1}, \ref{teo2}, \ref{teo3}, \ref{teo4} and \ref{teo5}, we can define the bounded linear operators
\begin{equation*}
\Lambda_i : \X \times \X \longrightarrow \HH_T \times \HH_T \qquad (i=1,2,3,4,5),
\end{equation*}
such that, for any $(u^0, v^0) \in \X$ and $(u^1, v^1) \in \X$,
$$\Lambda_i\left(  \left( \begin{array}{cc} u^0\\v^0 \end{array}\right), \left( \begin{array}{cc} u^1\\v ^1 \end{array}\right)\right  ):= \left( \begin{array}{cc} \vec{h}_i\\ \vec{g}_i \end{array}\right), $$
where $\vec{h}_i$ and $\vec{g}_i$ were defined in the Introduction.

\begin{proof}[Proof of Theorem \ref{main_theo}]
We treat the nonlinear problem \eqref{gg1}-\eqref{gg2} using a classical fixed point argument.

According to Remark \ref{integralform}, the solution can be written as
\begin{align*}
\left( \begin{array}{cc} u(t)\\v(t) \end{array}\right) =&  W_0(t) \left( \begin{array}{cc} u_0\\v_0 \end{array}\right) + W_{bdr}(t)\left(\begin{array}{cc}\vec{h}_i \\  \vec{g}_i   \end{array}\right)\\
&- \int_0^t W_0 (t-\tau)\left( \begin{array}{cc}  a_1 (vv_x)(\tau)+a_2 (uv)_x(\tau) \\  \frac{r}{c}v_x(\tau)+ \frac{a_2b}{c}(uu_x)(\tau)+\frac{a_1b}{c}(uv)_x(\tau) \end{array}\right) d\tau,
\end{align*}
for $i=1, 2, 3, 4, 5$, where $\{W_0(t)\}_{t\ge 0}$ and  $\{W_{bdr}(t)\}_{t\ge 0}$ are  the operators defined in Proposition \ref{prop1}. We only analyze the case $i=1$, since the other cases are analogous we will omit them.

For $u,v \in \ZZ_T$,  let us define
$$\left(\begin{array}{cc} \upsilon \\ \nu (T,u,v)\end{array}\right)  := \int_0^T W_0(T-\tau)   \left( \begin{array}{cc}  a_1  (vv_x)(\tau)  + a_2 (uv)_x (\tau) \\
 \frac{a_2b}{c}(uu_x)(\tau) + \frac{a_2b}{c} (uv)_x(\tau) \end{array} \right)  d\tau  $$
 and consider the map
\begin{align*}
\Gamma\left( \begin{array}{cc} u\\v \end{array}\right) = & W_0(t) \left( \begin{array}{cc} u^0\\v^0 \end{array}\right) + W_{bdr}(x)  \Lambda_1\left(    \left( \begin{array}{cc}   u^0\\v^0  \end{array}\right),  \left( \begin{array}{cc}  u^1 \\ v^1 \end{array}\right) +  \left( \begin{array}{cc} v \\  \nu(T,u,v) \end{array}\right) \right) \\
&- \int_0^t W_0 (t-\tau)
\left( \begin{array}{cc}  a_1 (vv_x)(\tau)+a_2 (uv)_x(\tau) \\  \frac{r}{c}v_x(\tau)+ \frac{a_2b}{c}(uu_x)(\tau)+\frac{a_1b}{c}(uv)_x(\tau) \end{array}\right) d\tau.
\end{align*}
If we choose
\begin{equation}\label{controli}
\left( \begin{array}{cc} \vec{h}_1 \\ \vec{g}_1 \end{array}\right)  = \Lambda_1\left(    \left( \begin{array}{cc}   u^0\\v^0  \end{array}\right),  \left( \begin{array}{cc}  u^1 \\ v^1 \end{array}\right) +  \left( \begin{array}{cc} v \\  \nu(T,u,v) \end{array}\right) \right),
\end{equation}
from Theorem \ref{teo3}, we get
$$\Gamma\left( \begin{array}{cc} u\\v \end{array}\right)\Big|_{t=0}= \left( \begin{array}{cc} u^0\\v^0\end{array}\right)$$
and
$$\Gamma\left( \begin{array}{cc} u\\v \end{array}\right)\Big|_{t=T}= \left( \begin{array}{cc} u^1\\v^1\end{array}\right)+\left(\begin{array}{cc} v \\ \nu (T,u,v)\end{array}\right)-\left(\begin{array}{cc} v \\ \nu (T,u,v)\end{array}\right)= \left( \begin{array}{cc} u^1\\v^1\end{array}\right).$$
Now we prove that the map $\Gamma$ is a contraction in an appropriate metric space, then its fixed point $(u,v)$ is the solution of \eqref{gg1}-\eqref{gg2} with $\vec{h}_1$ and $\vec{g}_1$ defined by \eqref{controli}, satisfying \eqref{exactcontrol'a}.  In order to prove the existence of the fixed point we apply the Banach fixed point theorem to the restriction of $\Gamma$ on the closed ball
$$B_r=\left\{ ( u, v)  \in  \ZZ_T : \left\| ( u, v )\right\|_{\ZZ_T} \le r \right\},$$
for some $r>0$.
\begin{itemize}
\item[(i)]\textit{$\Gamma$ maps $B_r$ into itself.}
\end{itemize}
Using Proposition \ref{prop2} there exists a constant $C_1>0$, such that

\begin{align*}
\left\|\Gamma\left( \begin{array}{cc} u\\v \end{array}\right) \right\|_{\ZZ_T}  \le&   C_1\left\lbrace  \left\|  \left( \begin{array}{cc} u^0\\v^0 \end{array}\right) \right\|_{\X}
 +\left\| \Lambda_1\left(    \left( \begin{array}{cc}   u^0\\v^0  \end{array}\right),  \left( \begin{array}{cc}  u^1 \\ v^1 \end{array}\right) +  \left( \begin{array}{cc} v \\  \nu (T,u,v)\end{array}\right) \right)\right\|_{\HH_T}
 \right\rbrace\\
&+C_1\left\lbrace\int_0^t \left\| \left( \begin{array}{cc}  a_1 (vv_x)(\tau)+a_2 (uv)_x(\tau) \\  \frac{r}{c}v_x(\tau)+ \frac{a_2b}{c}(uu_x)(\tau)+\frac{a_1b}{c}(uv)_x(\tau) \end{array}\right)\right\|_{\X} d\tau\right\rbrace.
\end{align*}
Moreover, since
\begin{multline*}
\left\| \Lambda_1\left(    \left( \begin{array}{cc}   u^0\\v^0  \end{array}\right),  \left( \begin{array}{cc}  u^1 \\ v^1 \end{array}\right) +  \left( \begin{array}{cc} v \\  \nu (T,u,v)\end{array}\right) \right)\right\|_{\HH_T} \leq C_2 \left\lbrace \left\| \left( \begin{array}{cc}   u^0\\v^0  \end{array}\right) \right\|_{\X}+ \left\|\left( \begin{array}{cc}  u^1 \\ v^1 \end{array}\right) \right\|_{\X} \right. \\
\left. + \left\| \left( \begin{array}{cc} v \\  \nu (T,u,v)\end{array}\right)\right\|_{\X} \right\rbrace,
\end{multline*}
applying Lemma \ref{lem3}, we can deduce that
\begin{equation*}
\left\|\Gamma\left( \begin{array}{cc} u\\v \end{array}\right) \right\|_{\ZZ_T}  \le   C_3\delta +C_4(r+1)r,
\end{equation*}
where $C_4$ is a constant depending only $T$. Thus, choosing $r$ and $\delta$ such that $$ r=2C_3\delta$$ and $$ 2C_3C_4\delta +C_4\leq \frac12,$$ the operator $\Gamma$ maps $B_r$ into itself for any $(u,v) \in \ZZ_T$.
\begin{itemize}
\item[(ii)]\textit{$\Gamma$ is contractive.}
\end{itemize}
Proceeding as the proof of Theorem \ref{nonlinearteo}, we obtain
\begin{equation*}
\left\|\Gamma\left( \begin{array}{cc} u\\v \end{array}\right) - \Gamma\left( \begin{array}{cc} \widetilde{u}\\\widetilde{v} \end{array}\right) \right\|_{\ZZ_T}  \le   C_5(r+1)r \left\|\left( \begin{array}{cc} u-\widetilde{u}\\v-\widetilde{v} \end{array}\right) \right\|_{\ZZ_T},
\end{equation*}
for any $(u,v), (\widetilde{u},\widetilde{v}) \in B_r$ and a constant $C_5$ depending only on $T$. Thus, taking $\delta >0$, such that $$\gamma= 2C_3C_5\delta +C_5 < 1,$$ we obtain
\begin{equation*}
\left\|\Gamma\left( \begin{array}{cc} u\\v \end{array}\right) - \Gamma\left( \begin{array}{cc} \widetilde{u}\\\widetilde{v} \end{array}\right) \right\|_{\ZZ_T}  \le   \gamma \left\|\left( \begin{array}{cc} u-\widetilde{u}\\v-\widetilde{v} \end{array}\right) \right\|_{\ZZ_T}.
\end{equation*}
Therefore, the map $\Gamma$ is a contraction. Thus, from (i), (ii) and the Banach fixed point theorem, $\Gamma$ has a fixed point in $B_r$  and its fixed point is the desired solution. The proof of Theorem \ref{main_theo} is, thus, complete.
\end{proof}

\noindent\textbf{Acknowledgments:}
Roberto A. Capistrano--Filho was supported by CNPq (Brazil), Project PDE, grant 229204/2013-9, Fernando A. Gallego was supported by CAPES (Brazil) and Ademir F. Pazoto was partially supported by CNPq (Brazil).


\begin{thebibliography}{99}

\bibitem{BoPoSaTo}
J. L. Bona, G. Ponce, J.-C. Saut and M. M. Tom, A Model System for Strong Interaction Between Internal Solitary Waves, {\em Commun. Math. Phys.} 143 (1992), 287--313.

\bibitem{bonasunzhang2003}
J.J. Bona, S.M. Sun and B.-Y. Zhang,  A nonhomogeneous boundary-value problem for the Korteweg-de Vries equation posed on a finite domain, {\em Comm. Partial Differential Equations}, 28 (2003), 1391--1438.

\bibitem{caicedo_caspistrano_zhang_2015}
M. Caicedo, R. A. Capistrano--Filho and Bingyu Zhang,  Neumann boundary controllability of the Korteweg-de Vries equation on a bounded domain, {\em arXiv preprint}, (2015), arXiv:1508.07525.


\bibitem{CaZhaSun}
R. A. Capistrano--Filho, S. M. Sun and B.-Y. Zhang, General boundary value problems of the Korteweg-de Vries equation on a bounded domain. {\em Preprint}, (2016).

\bibitem{cerpapazoto2011}
E. Cerpa and A. F. Pazoto, A note on the paper On the controllability of a coupled system
of two Korteweg-de Vries equations, {\em Commun. Contemp. Math}, 13 (2011), 183--189.

\bibitem{cerizh} E. Cerpa, I. Rivas and B.-Y. Zhang, Boundary controllability of the Korteweg-de Vries equation on a bounded domain, {\em SIAM J. Control Optim.}, 51 (2013), 2976--3010.


\bibitem{dolecki} S. Dolecki and D. L. Russell, A general theory of observation and control, {\em SIAM J. Control Optimization}, 15 (1977), 185--220.

\bibitem{geargrimshaw1984}
J. A. Gear and R. Grimshaw,  Weak and strong interactions between internal solitary waves, {\em Studies in Appl. Math}, 70 (1984),  235-258.


\bibitem{lions} J.-L. Lions, {\em Contr\^olabilit\'e Exacte, Perturbations et Stabilisation de Syst\`emes Distribu\'es,}
Tome 1, Masson, Paris, (1988).

\bibitem{rosier} L.  Rosier,\, Exact  boundary controllability for the Korteweg-de Vries  equation on a bounded domain,  {\em ESAIM Control
Optim. Cal. Var}, 2 (1997), 33-55.

\bibitem{MiOr}
S. Micu and J. H. Ortega, On the controllability of a linear coupled system of Korteweg--de Vries equations, {\em in Mathematical and Numerical Aspects of Wave Propagation }(Santiago de Compostela, 2000) (SIAM, Philadelphia, PA, 2000), 1020--1024.

\bibitem{micuortegapazoto2009}
S. Micu, J. Ortega and A. Pazoto, On the Controllability of a Coupled system of two Korteweg-de Vries equation,  {\em Commun. Contemp. Math},  11 (5) (2009), 779--827.


\bibitem {KePoVe}
C. E. Kenig, G. Ponce and L. Vega, Well-posedness of the initial value problem for the Korteweg--de Vries equation,{\em J. Amer. Math. Soc.}, 4 (1991), 323--347.

\bibitem{SaTz} J.-C. Saut and N. Tzvetkov, On a model system for the oblique interaction of internal gravity waves, {\em M2AN Math. Model. Numer. Anal.}, 34 (2000), 501--523.

\bibitem{yosida} K. Yosida, {\em Functional Analysis}, Springer-Verlag, Berlin Heidelberg New York, (1978).


\end{thebibliography}
\end{document}